\documentclass[11pt]{article}
\usepackage[utf8]{inputenc}
\usepackage{tikz}
\usepackage{hyperref,verbatim}
\usepackage{enumitem,bm}
\usepackage{todonotes}
\usepackage[font=footnotesize,labelfont=bf]{caption}
\usepackage{setspace}

\usepackage{pgf,tikz,subcaption}
\usetikzlibrary{arrows,shapes}
\usetikzlibrary{decorations.pathreplacing}
\usepackage{fullpage}

\usepackage{amsthm,amsmath,amssymb}
\usepackage{url}
\usepackage[capitalise]{cleveref}

\newtheorem{conj}{Conjecture}

\newtheorem{theorem}[conj]{Theorem}
\newtheorem{lem}[conj]{Lemma}
\newtheorem{cor}[conj]{Corollary}

\newtheorem{obs}[conj]{Observation}
\newtheorem{claim}{Claim}

\newtheorem*{mainresredux}{Main Results}

\newtheorem*{keylempr}{Key Lemma, precise version}

\theoremstyle{definition}
\newtheorem*{remark}{Remark}
\newtheorem{defi}[conj]{Definition}
\newtheorem{example}[conj]{Example}
\newtheorem{question}[conj]{Question}

\newcommand\dist{\textrm{dist}}
\newcommand\vecC{\vec{C}}
\newcommand\vecD{\vec{D}}

\newcommand*{\myproofname}{Proof}
\newenvironment{claimproof}[1][\myproofname]{\begin{proof}[#1]}{\end{proof}}

\DeclareMathOperator{\diam}{diam}

\renewcommand{\a}{\alpha}
\renewcommand{\b}{\beta}

\renewcommand{\c}{\gamma}

\newcommand{\C}{\mathcal{C}}
\newcommand{\Co}[2]{\C(#1,#2)}
\newcommand{\CLG}{\Co{G}{L}}
\newcommand{\CkG}{\Co{G}{k}}
\newcommand{\hC}{\widehat{\C}}
\newcommand{\hCo}[2]{\hC(#1,#2)}
\newcommand{\hCLG}{\hCo{G}{L}}
\newcommand{\pl}[1]{$+#1$-list}
\newcommand{\pla}[1]{$+#1$-list-assignment}

\tikzstyle{uStyle}=[shape = circle, minimum size = 6.5pt, inner sep = 0pt,
outer sep = 0pt, draw, fill=white, semithick]
\tikzstyle{sStyle}=[shape = rectangle, minimum size = 4.5pt, inner sep = 0pt,
outer sep = 0pt, draw, fill=white, semithick]
\tikzstyle{lStyle}=[shape = circle, minimum size = 4.5pt, inner sep = 0pt,
outer sep = 0pt, draw=none, fill=none]
\tikzset{every node/.style=uStyle}

\title{Reconfiguration of List Colourings}
\author{Stijn Cambie\,\footnote{\,Department of Computer Science, KU Leuven Campus Kulak-Kortrijk, Belgium.
Supported by a FWO grant with grant number 1225224N.
Email: \href{mailto:stijn.cambie@hotmail.com}{\texttt{stijn.cambie@hotmail.com}}.}\qquad
Wouter Cames van Batenburg\,\footnote{\,D\'epartement d'Informatique, Universit\'e libre de Bruxelles, Belgium. Supported by the Belgian National Fund for Scientific Research (FNRS).
Email: \href{mailto:w.p.s.camesvanbatenburg@gmail.com}{\texttt{w.p.s.camesvanbatenburg@gmail.com}}.}\qquad
Daniel W. Cranston\,\footnote{\,Department of Computer Science, Virginia Commonwealth University, Richmond, Virginia, USA.
Email: \href{mailto:dcransto@gmail.com}{\texttt{dcransto@gmail.com}}.}\\[1mm]
Jan van den Heuvel\,\footnote{\,Deptartment of Mathematics, London School of Economics \& Political Science, London, UK.
Email: \href{mailto:j.van-den-heuvel@lse.ac.uk}{\texttt{j.van-den-heuvel@lse.ac.uk}}.}\qquad Ross J. Kang\footnote{\,Korteweg--de Vries Institute for Mathematics, University of Amsterdam, the Netherlands.
Partially supported by grant OCENW.M20.009 and the Gravitation Programme NETWORKS (024.002.003) of the Dutch Research Council (NWO).
Email: \href{mailto:r.kang@uva.nl}{\texttt{r.kang@uva.nl}}.}}
\date{}

\begin{document}

\maketitle

\begin{abstract}
    \noindent
Given a proper (list) colouring of a graph $G$, a recolouring step changes the colour at a single vertex to another colour (in its list) that is currently unused on its neighbours, hence maintaining a proper colouring.  Suppose that each vertex $v$ has its own private list~$L(v)$ of allowed colours such that $|L(v)|\ge \deg(v)+1$.
We prove that if $G$ is connected and its maximum degree $\Delta$ is at least $3$, then for any two proper $L$-colourings in which at least one vertex can be recoloured, one can be transformed to the other by a sequence of $O(|V(G)|^2)$ recolouring steps.
We also show that reducing the list-size of a single vertex $w$ to $\deg(w)$ can lead to situations where the space of proper $L$-colourings is `shattered'.  Our results can be interpreted as showing a sharp phase transition in the Glauber dynamics of proper $L$-colourings of graphs.

This constitutes a `local' strengthening and generalisation of a result of Feghali, Johnson, and Paulusma, which considered the situation where the lists are all identical to $\{1,\ldots,\Delta+1\}$.
\end{abstract}

\section{Introduction}

One of the most basic results in graph theory is that the vertices of a graph $G$ can be properly coloured with at most $\Delta+1$ colours, where $\Delta$ is the maximum degree of $G$.
This remains true if instead we provide every vertex $v$ with a private \emph{list} $L(v)$ of at least $\deg(v)+1$ colours to choose from (where $\deg(v)$ denotes the number of neighbours of $v$).
This means that the collection of proper colourings is non-empty.
What does this collection look like?
We are especially interested in how `closely linked' its colourings are to each other.
For the situation above, we show that a vast majority of all colourings are connected, in a precise way that we will discuss below.

Connectedness of \emph{all} colourings implies that a widely-studied Markov chain known as the \emph{Glauber dynamics} of (list) colourings is ergodic.
Our results demonstrate a fine balance essentially at the boundary of ergodicity; increasing or decreasing the number of available colours by one at even a single vertex changes the guaranteed component structure dramatically. 

To state our main results, we need some basic definitions.
Given a graph $G$, a \emph{list-assignment~$L$} is an assignment of lists of colours $L(v)\subseteq\mathbb{N}$ to each vertex $v$ of $G$.
A proper \emph{$L$-colouring} is a proper colouring $\c:V(G)\rightarrow \mathbb{N}$ such that $\c(v)\in L(v)$ for every $v\in V(G)$.
(We often omit the adjective `proper' when this does not lead to confusion.)
A (single vertex) \emph{recolouring} of a proper $L$-colouring~$\c$ changes the colour under $\c$ of one vertex $v$ to another colour of $L(v)$, such that the resulting colouring is again proper.
Such a recolouring is only possible if the new colour is not already used by~$\c$ on any of the neighbours of $v$.
A \emph{recolouring sequence} is a sequence of single vertex recolourings.
If no vertex of $G$ can be recoloured from $\c$, we call $\c$ \emph{frozen}; otherwise it is \emph{unfrozen}.

Given a graph $G$ and list-assignment $L$, the \emph{list colouring reconfiguration graph $\CLG$} has as vertices all proper $L$-colourings of $G$, and two $L$-colourings are adjacent if one can be obtained from the other by recolouring a single vertex.
The frozen $L$-colourings form singletons in $\CLG$ (components with just one vertex).
As our main focus is on the recolouring components of $L$-colourings, we may restrict our attention to the unfrozen $L$-colourings, and thus to the subgraph $\hCLG$ of $\CLG$ induced by such colourings.
If for all $v\in V(G)$ we have $L(v)=\{1,\ldots,k\}$ for some positive integer $k$, then we write $\CkG$ and~$\hCo{G}{k}$.

\begin{mainresredux}
	Let $G$ be a connected graph with $n$ vertices, and let $L$ be a list-assignment of $G$.
	\begin{list}{}{%
		\setlength{\parsep}{1pt}
		\setlength{\topsep}{2pt}
		\setlength{\labelsep}{2mm}
		\setlength{\labelwidth}{0.5\parindent}
		\setlength{\leftmargin}{\parindent}\setlength{\itemindent}{0pt}
		\setlength{\listparindent}{0pt}}
		\item[{\rm 1.}]\label{itm:keylemma}
		{\rm (\textbf{Key Lemma})}\quad
        If $|L(v)|\ge\deg(v)+1$ for all $v\in V(G)$ and $|L(w)|\ge\deg(w)+2$ for at least one $w\in V(G)$, then $\CLG$ is connected and has diameter $O(n^2)$.
		\item[{\rm 2.}]\label{itm:mainthm}
		{\rm (\textbf{Main Theorem})} \quad
        If $L|(v)|\ge\deg(v)+1$ for all $v\in V(G)$ and $G$ has maximum degree $\Delta\ge3$, then $\hCLG$ is connected and has diameter $O(n^2)$.
%
		Moreover, if $\Delta=o(n^{1/4})$, then the number of frozen $L$-colourings is vanishingly small compared to the number of unfrozen ones.
		\item[{\rm 3.}]\label{itm:shatter}
		{\rm (\textbf{Shattering Observation})}\quad
		If $|L(w)|=\deg(w)$ for some $w\in V(G)$ and $|L(v)|\ge\deg(v)+1$ for all other $v\in V(G)$, then
        the number of components in $\hCLG$ can be exponential (in terms of~$n$).
	\end{list}
\end{mainresredux}

These results align with several active research areas in graph theory, theoretical computer science, and statistical physics.
Before going into details about those connections, we give a few observations about the results above.

The condition that $G$ is connected is essential in the Key Lemma, as otherwise we can have $L$-colourings that are frozen in some components of $G$ (and hence can't be recoloured on those vertices), but unfrozen in other components.

The condition that the maximum degree is at least $3$ in the Main Theorem is necessary.
For instance, the thirty $3$-colourings of the cycle $C_5$ are all unfrozen, but $\Co{C_5}{3}$ has two connected components with  fifteen $3$-colourings each.
In fact, all cycles of length at least $5$ and paths of length at least $6$ show similar behaviour; 
for more on this,
see Subsection~\ref{sec4.1}.

The two diameter bounds are sharp up to a multiplicative constant, even when $G$ is restricted to being a tree.
This is the case as any subpath of degree-$2$ vertices of length $\Omega(n)$ is enough to ensure that $\hCLG$ has  $\Omega(n^2)$ diameter.
We remark here that a subset of the authors~\cite{CCvBC} showed that with lists of size $\deg(v)+2$ for each vertex $v$, the diameter of $\C(G,L)$ is at most $2n$.

In the final section,
we discuss further conditions that could strengthen our results.

\medskip
In `classical' chromatic graph theory we are typically interested in sharp  conditions that ensure that $G$ is $k$-colourable; i.e., that $\CkG$ is non-empty.
As mentioned above, a basic example is the condition $k\ge\Delta+1$, which is sharp when $G$ is a clique or an odd cycle.
By Brook's Theorem~\cite{brooks}, these are the only cases for connected $G$ where we actually need $\Delta+1$ colours.

Arguably, `classical' graph colouring reconfiguration theory has followed a similar path, through a focus instead on sharp conditions that ensure that $\CkG$ is connected.
Jerrum~\cite{jerrum} showed that $\CkG$ is connected whenever $k\ge \Delta+2$.
This result follows from a simple inductive proof, as we elaborate upon later.

By contrast, there exists a connected~$G$ with a $(\Delta+1)$-colouring that admits no recolouring step; and then, by permutation of the colours, $\Co{G}{\Delta+1}$ must contain at least $(\Delta+1)!$ singleton components.
The simplest examples are cliques and cycles with length divisible by $3$, but there are many others, e.g.\ the $3$-cube $Q_3$.
More generally, take any $(\Delta+1)$-partite graph in which each vertex has exactly one neighbour in each part, but its own; when we use the same colour on all vertices in each part, no recolouring step is possible.
For connected $G$, is having a frozen $(\Delta+1)$-colouring also a necessary condition for $\Co{G}{\Delta+1}$ to be disconnected? Feghali, Johnson, and Paulusma~\cite{FJP} showed that this is indeed the case, and in an even stronger sense as follows.

\begin{theorem}[Feghali, Johnson, and Paulusma \cite{FJP}]
\label{thm:fjp}
	If $G$ is a connected graph with $n$ vertices and maximum degree $\Delta\ge3$,
	then $\hCo{G}{\Delta+1}$ is connected and has diameter $O(n^2)$.
\end{theorem}

Determining whether an arbitrary connected $\Delta$-regular $G$ has a frozen $(\Delta+1)$-colouring appears to be non-trivial. 
However, this characterisation for disconnectedness of $\Co{G}{\Delta+1}$ given in \Cref{thm:fjp} can be viewed as a reconfiguration parallel to Brooks' simple characterisation for emptiness of $\Co{G}{\Delta}$.
The diameter bound of $O(n^2)$ was recently improved by Bousquet, Feuilloley, Heinrich, and Rabie~\cite{BFHR} to $O(n)$ (albeit with a leading constant factor that hides a superexponential dependency upon~$\Delta$).
As mentioned earlier, a subset of the present authors~\cite{CCvBC} recently showed a $2n$ bound on the diameter of $\Co{G}{\Delta+2}$ (improving considerably on the bound implicit in Jerrum's original argument), and conjectured a~$1.5n$ diameter bound.

Our work draws the same parallel, but in the refined setting where each vertex gets a private or `local' list of colours.
Colouring greedily shows the following.
If $G$ is connected and $L$ is a list-assignment such that $|L(v)|\ge \deg(v)$ for all $v\in V(G)$ and $|L(w)|\ge \deg(w)+1$ for at least one $w\in V(G)$, then $G$ has an $L$-colouring.
The Key Lemma constitutes a reconfiguration parallel of this observation.
Furthermore, a classic result independently due to Borodin~\cite{borodin-list} and Erd\H{o}s, Rubin, and Taylor~\cite{ERT} states that if $G$ is connected and $L$ is a list-assignment such that$|L(v)|\ge \deg(v)$ for all $v\in V(G)$, then $G$ has an $L$-colouring unless $G$ has a specific structure (a mild generalisation of a tree called a \emph{Gallai tree}) together with a specific list-assignment $L$.
The Main Theorem constitutes a reconfiguration parallel of this `local' Brooks'-type characterisation.

Following on the work in~\cite{borodin-list,ERT}, there has been extensive research in chromatic graph theory for this notion of `local' colouring: for triangle-free graphs \cite{BKNP22,DJKP20}, Reed's Conjecture \cite{CKPS13,Kin09thesis}, disjoint list colourings~\cite{CCDK24}, and fractional colourings~\cite{KP24}, to name a few examples.

\medskip
One inspiration for research on recolouring graphs comes from the theory of approximate counting, which has deep links via randomised algorithms to statistical physics.
A unifying topic is the Glauber dynamics for proper colourings of a graph.
In physical applications, the graph usually assumes a lattice structure, we consider the colours as spins in the system, and take an idealised extreme temperature regime, so that no two neighbouring vertices (sites in the lattice crystal) may share the same colour (spin).

In the Glauber dynamics of proper $k$-colourings of a graph $G$, at each step some vertex $v\in V(G)$ and some colour $c\in\{1,\ldots,k\}$ are chosen uniformly at random.
If $v$ can be recoloured with colour~$c$ such that the colouring remains proper, then the Markov chain takes this colouring as its new state; otherwise it remains unchanged.
As mentioned earlier, Jerrum~\cite{jerrum} noticed that if $k\ge \Delta+2$, then the state space of this chain is connected.
Since this chain is also aperiodic and symmetric, it converges to an equilibrium that equals the uniform distribution over all proper $k$-colourings,  by the ergodic theorem; see~\cite{Jer03book}.
This probability distribution is an important global parameter of the system, since it is equivalent to the number of proper $k$-colourings.
Jerrum~\cite{jerrum} moreover conjectured that this process converges rapidly (within a number of steps polylogarithmic in the number of states) to equilibrium, which would imply an efficient randomised approximation scheme for estimating the number of proper $(\Delta+2)$-colourings; \cite{FrVi07} provides wider context.
Importantly, Jerrum~\cite{jerrum} (see also~\cite{SaSo97}) further showed this conjecture to be true up to a constant factor increase in $k$.
Significant effort has gone towards reducing the constant factor~\cite{CaVi25,CDMPP19,Vig00}, even if proving rapid mixing for all $k\ge\Delta+2$ still seems far out of reach.

Theorem~\ref{thm:fjp} may be interpreted as follows.
For connected $G$, if $k\ge \Delta+1$, then the Glauber dynamics of proper $k$-colourings is ergodic, provided the chain is restricted to the unfrozen colourings.
Bonamy, Bousquet, and Perarnau~\cite{BBP21} showed moreover that the frozen colourings constitute a vanishingly small proportion of all proper $k$-colourings.
This means that, provided we check that the initial colouring is unfrozen (which is the case nearly always), the Glauber dynamics converges to an almost uniform sample from the space of proper $k$-colourings.
Note that the additional assertion in our Main Theorem is a partial strengthening of this result of Bonamy {et al}.
Furthermore, the fact that the diameter of the unfrozen part of the state space is $O(n)$~\cite{BFHR} lends credence to the idea that the rapid mixing conjecture might still hold with $\Delta+1$ colours, provided we restrict to the unfrozen colourings.
Note that Bonamy et al.~\cite{BBP21} have also shown that the mixing time for this problem is $\Omega(n^2)$. 

For the obvious generalisation of the Glauber dynamics to list colourings, at each step a vertex~$v$ and a colour $c\in L(v)$ are chosen uniformly at random.
Again, the Markov chain takes the new colouring as its new state as long as it is proper; otherwise it remains unchanged.
Our Main Results may be viewed in terms of the Glauber dynamics of proper $L$-colourings of some connected~$G$, provided the list-size conditions for $L(v)$ are $\deg(v)+1$ for all vertices $v$ except possibly for one.
Our Main Theorem says that the Glauber dynamics on unfrozen $L$-colourings is ergodic provided all of the list-sizes are at least $\deg(v)+1$ for all vertices $v$ and the maximum degree $\Delta$ is at least~$3$; and the rest of the state space is negligible if $\Delta$ is not too large.
Our Key Lemma says that there are no frozen $L$-colourings if in addition we require one of the vertices to have a strictly larger list-size.
Our Shattering Observation demonstrates that if even merely one vertex $v$ has a smaller list-size of $\deg(v)$, then the state space may fracture into many disconnected pieces.
In such a situation, a more `global' dynamics --- one that can recolour more than one vertex at a time --- is necessary to approximately uniformly sample from the proper $L$-colourings.

To underscore the `local' perspective we have adopted in this work, specific to the Glauber dynamics of colourings, we propose the following strengthening of the rapid mixing conjecture.

\begin{conj}
\label{conj:rapid}
    Let $G$ be a connected graph, and let $L$ be a list-assignment of $G$.
    If $|L(v)|\ge\deg(v)+2$ for all $v\in V(G)$, then the Glauber dynamics for proper $L$-colourings of $G$ mixes rapidly.
\end{conj}

In a sense, this conjecture has already been considered in the (randomised) algorithms community.
More specifically, many of the approaches towards the uniform version of the conjecture~\cite{CaVi25,CDMPP19,Vig00} also seem to yield local analogues.
Moreover, Chen, Liu, Manu, and Moitra~\cite{CLMM23} have shown that Conjecture~\ref{conj:rapid} holds when the lists are just slightly larger, of size $\deg(v)+3$ for each vertex $v$, and $G$ has sufficiently large girth.
Indeed, local colouring has arisen in various algorithmic contexts (and naturally so, given how efficient colouring algorithms are often based upon local arguments), not just in approximate counting~\cite{CLMM23,GKM15}, but also in distributed computing \cite{CPS17,EJM23,HKNT22} and graph sparsification~\cite{AlAs20,FGHKN24}, for instance.

\subsection{Proof Ideas}

In this subsection we outline some of the key ideas used to prove our results.

\begin{remark}
In the remainder of the paper, unless otherwise stated, we assume that all graphs encountered are connected and all (list) colourings are proper.
\end{remark}

Given a list-assignment $L$ of a graph $G$ and an $L$-colouring $\c$, we say that a vertex $v$ is \emph{frozen under $\c$} if $v$ cannot be recoloured (i.e., if all colours in $L(v)\setminus\{\c(v)\}$ appear among the neighbours of $v$); otherwise it is \emph{unfrozen}.
Following the earlier definition, this means a colouring~$\c$ is frozen if all vertices are frozen under $\c$, and $\c$ is unfrozen if at least one vertex is unfrozen.

Informally, an unfrozen vertex remains unfrozen when we restrict a list-assignment to an induced subgraph containing it.
We formalize this intuition in the following observation, which we use frequently and implicitly throughout the rest of the paper.

\begin{obs}
\label{lem:restrictedontosubgraph}
    Let $G$ be a graph with a list-assignment $L$, and let $\a$ be an $L$-colouring of $G$.
    Given a set of vertices $W\subseteq V(G)$, we form a list-assignment $L_W$ for $G[W]$ from $L$ by removing the colours of the neighbours, i.e., $L_W(v):=L(v)\setminus \bigcup_{w\in N_G(v)\setminus W} \a(w)$ for all $v\in W$.
    Let $\a_W$ denote the $L_W$-colouring formed by restricting $\a$ to $G[W]$.
    Fix $x,y \in W$.
    If $y$ is unfrozen under~$\a$ then it is unfrozen under $\a_W$.
    Moreover, if $\a_W$ can be recoloured to an $L_W$-colouring of $G[W]$ such that both $x$ and $y$ are unfrozen, then the same recolouring sequence (lifted to $G$)  recolours $\a$ to an $L$-colouring of~$G$ such that $x$ and $y$ are unfrozen.
\end{obs}

\begin{proof}
If $y$ is unfrozen under $\a$, then some colour $c \in L(y)$ is unused by $\a$ on $N_G[y]$.
Therefore $c\in L_W(y)$, so $y$ is unfrozen under $\a_W$.

For the second statement, note that every $L_W$-colouring of $G[W]$ is also an $L$-colouring of $G[W]$, and every vertex that is unfrozen under some $L_W$-colouring~$\c$ of $G[W]$ is also unfrozen under the $L$-colouring of $G$ formed by the union of $\c$ and $\a|_{V(G)-W}$.
\end{proof}

Most of our arguments do not use the actual contents of the lists in a given list-assignment, but only the size of the lists, and in particular how much larger a particular list $L(v)$ is compared to the degree of the vertex $v$.
To help describe this more compactly, given some non-negative integer $t$ and a list-assignment $L$, we say that a vertex $v$ has a \emph{\pl{t}} (with respect to~$L$) if $|L(v)| \ge \deg(v)+t$, and that $L$ is a \emph{\pla{t}} if every vertex in $G$ has a \pl{t}.

We motivate our methodology by describing a useful step towards understanding the characterisation for  degree-choosability mentioned earlier (see~\cite{borodin-list,ERT}).
That is, for which graphs~$G$ is there always guaranteed to be an $L$-colouring for any given \pla{0} $L$ of $G$?
We can easily prove that if $L$ is instead a \pla{0} of $G$ such that some  $v\in V(G)$ has a \pl{1}, then $\CLG$ is non-empty.
This is by induction on the number of vertices.
If $L_{-v}$ denotes the restriction of $L$ to $G-v$, then note that any neighbour of $v$ has a \pl{1} with respect to $L_{-v}$.
In particular, for any component of $G-v$, $L_{-v}$ is a \pla{0} in which at least one vertex has a \pl{1}.
Thus, by the inductive hypothesis, $\Co{G-v}{L_{-v}}$ is non-empty.
Since~$v$ has a \pl{1} with respect to $L$, any element of $\Co{G-v}{L_{-v}}$ can be extended by colouring $v$ with an available colour from~$L(v)$ to an $L$-colouring of $G$, completing the induction.

Along the same parallel we proposed earlier on page~\pageref{thm:fjp}, we can obtain with the same idea a weaker form of the Key Lemma, one where we are not concerned by a bound on the diameter.
Again this is by induction on the number of vertices.
Suppose that $L$ is a \pla{1} of~$G$ such that some vertex $v$ has a \pl{2}.
If $L_{-v}$ denotes the restriction of $L$ to $G-v$, then note that any neighbour $u$ of $v$ has a \pl{2} with respect to $L_{-v}$.
In particular, for any component of $G-v$, $L_{-v}$ is a \pla{1} in which at least one vertex has a \pl{2}.
Thus, by the inductive hypothesis applied on each component, $\Co{G-v}{L_{-v}}$ is connected.
Now let $\a$ and~$\b$ be two elements of $\CLG$.
By the arguments above, their restrictions $\a_{-v}$ and $\b_{-v}$ to $G-v$ are connected in $\Co{G-v}{L_{-v}}$ by an $L_{-v}$-recolouring sequence.
We essentially `extend' this to an $L$-recolouring sequence between $\a$ and $\b$, except that we often may have to recolour~$v$ to another available colour in $L(v)$ to avoid some conflict on one of its neighbours before taking the next step in the $L_{-v}$-recolouring sequence.
This is always possible because $v$ has a \pl{2} with respect to~$L$.
This argument yields a bound on $\diam(\CLG)$ that is exponential in $|V(G)|$.
Obtaining a quadratic bound requires considerably more care, and will be done in the next section.

\begin{keylempr}
	Let $G$ be a connected graph with $n$ vertices, and let $L$ be a list-assignment of $G$.
	If $|L(v)|\ge \deg(v)+1$ for all $v\in V(G)$ and $|L(w)|\ge \deg(w)+2$ for at least one $w\in V(G)$, then $\CLG$ is connected and has diameter at most $\tfrac12(3n^2+5n)$.
\end{keylempr}

\medskip
If $G$ is $2$-connected but does not have all lists identical, then our Main Theorem essentially follows from our Key Lemma.
We find neighbours $v,w$ with $L(v)\not\subseteq L(w)$, colour $v$ from $L(v)\setminus L(w)$, and apply the Key Lemma to $G-v$.
If $G$ has cut-vertices, then we recolour one `outermost' block (a leaf in the block tree) from $\a$ to $\b$, and finish by induction on the number of blocks.
This brings us to the following.

An important special case of our Main Theorem is when $G$ is $2$-connected and regular and $L(v)=\{1,\ldots,\Delta+1\}$ for all $v$.
(This is the $2$-connected case in the main result of Feghali et al.~\cite{FJP-brooks}.)
In his elegant proof of Brooks' Theorem, Lov\'{a}sz~\cite{lov1975} proved the following easy lemma:
If $G$ is $2$-connected with minimum degree at least $3$, then there exist
$v,w_1,w_2\in V(G)$ with $w_1,w_2\in N(v)$ and $w_1w_2\notin E(G)$ such that $G-\{w_1,w_2\}$ is connected.
To prove Brooks' Theorem, assign the same colour to $w_1$ and $w_2$, then colour greedily inward along a spanning tree $T$ in $G-\{w_1,w_2\}$ rooted at $v$.

To handle the special case above, we elaborate on this idea.
First, we show that from every unfrozen $(\Delta+1)$-colouring we can get to one that uses a common colour $c_1$ on some $w_1,w_2$ as above.
Second, we show that we can `transfer' from the class of $(\Delta+1)$-colourings that use $c_1$ on $w_1,w_2$ (neighbours of $w$), which are reachable from $\a$, to the class of $(\Delta+1)$-colourings that use some $c_1'$ on $w_1',w_2'$ (neighbours of some $w'$), which are reachable from $\b$.
To effect this transfer, we construct a colouring that lies in both classes.

\medskip
Our proof of the Main Theorem is self-contained.
So, in particular, it provides an alternative proof of the main result of Feghali et al.~\cite{FJP}.
In fact, if we aim only to recover that the recolouring diameter of $\hCo{G}{\Delta+1}$ is $O(|V(G)|^2)$, then we can significantly streamline the proof, essentially needing only the contents of Sections \ref{key-lem-sec} and \ref{regular-sec}, as well as Claims~\ref{clm2B} and~\ref{clm3B} in the proof of \Cref{connected-thm}.

We also point out that even when all lists have size $\Delta(G)+1$, which is equivalent to the case when~$G$ is $\Delta$-regular, our Main Theorem is a new result.

\subsection{Organisation}

The remainder of the paper is organised as follows.
In the next section, we prove the precise form of the Key Lemma.
Section~\ref{regular-sec} deals with the special case that $G$ is a $2$-connected regular graph in which all lists are identical.
The Main Theorem is proved in Section~\ref{connect-sec}.
That section starts with a tedious but necessary analysis of some simple cases, including paths and cycles.
Although paths and cycles do not satisfy the condition in the Main Theorem that the maximum degree is at least~3, we may encounter them when `breaking down' the block tree of graphs with connectivity $1$. 
Finally, in \Cref{sec:sharpness} we describe the examples for the Shattering Observation, pose some questions, and discuss possible directions for future research.
In particular, we discuss extending our results to \emph{correspondence colourings}, which appears surprisingly non-straightforward.

\section{Proof of the Key Lemma}
\label{key-lem-sec}

Let $G$ be a connected graph, and let $L$ be a \pla{1} of $G$ with some $v\in V(G)$ where $|L(v)|\ge \deg(v)+2$.
The next three lemmas allow us to recolour from an arbitrary $L$-colouring to one where a given vertex $w$ uses any specified colour $\b(w)$ in its list;  applying this repeatedly yields the Key Lemma.
The first lemma allows us to `push' unfrozenness from one vertex $v$ to another vertex $w$ along a shortest $v,w$-path.
The second lemma allows us to recolour so that either $w$ uses~$\b(w)$, or else $\b(w)$ is used on a single neighbour $x$ of $w$ and both $w$ and $x$ are frozen.
The third lemma allows us to unfreeze $w$ or $x$ to move past the possible `bad' outcome of the second lemma, and ultimately recolour $w$ with $\b(w)$.

\begin{lem}
\label{lem:unfreezeOneVertex_listsize_deg+2}
    Let $G$ be a connected graph, $L$ be a \pla{1} for $G$, and $\a$ be an $L$-colouring.
    Fix $v,w\in V(G)$, let $d:=\dist(v,w)$, and let $P$ be a shortest $v,w$-path.
    If $v$ is unfrozen in $\a$, then there exists a recolouring sequence $\mathcal{S}$ from $\a$ to some
    $L$-colouring $\c$ such that
    (i) $w$ is unfrozen under $\c$,
    (ii) every vertex that is recoloured during $\mathcal{S}$ belongs to $V(P)\setminus \{w\}$, and
    (iii) $|\mathcal{S}| \leq d$.
    Moreover, note that if $|L(v)|\geq \deg(v)+2$, then $v$ is also unfrozen under $\c$.
\end{lem}

\begin{proof}
    Let $P=v_0\cdots v_d$, with $v_0=v$ and $v_d=w$.
    Let $i$ be the largest index with $v_i$ unfrozen; $i$ exists, since $v_0$ is unfrozen.
    Recolouring $v_i$ unfreezes $v_{i+1}$, and we finish by induction on $d-i$.
\end{proof}

\begin{lem}
\label{lem:recolouringToUniqueBadNeighbour}
    Let $G$ be a connected graph, and let $L$ be a \pla{1} for $G$.
    Let $\a, \b$ be two unfrozen $L$-colourings.
    For each vertex $w$, there exists a recolouring sequence $\mathcal{S}$ from $\a$ to some $L$-colouring $\c$ such that (i) either $\c(w)=\b(w)$, or there is precisely one vertex $x\in N(w)$ with $\c(x)=\b(w)$ and $w$ and $x$ are frozen, and (ii) $|\mathcal{S}|\leq  \deg(w)+2$, with equality only if $\c(w)=\b(w)$.
\end{lem}

\begin{proof}
    Given $\a$ and $w$, a \emph{bad neighbour} is any neighbour of $w$ coloured $\b(w)$.
    We assume that $\a(w)\neq \b(w)$ and that $w$ has at least one bad neighbour under $\a$; otherwise we recolour $w$, if needed, and are done.
    Let $B_1$ denote the set of unfrozen bad neighbours, and $B_2$ the set of frozen bad neighbours.
    We first recolour each bad neighbour in $B_1$.
    Note that $B_1\cup B_2$ forms an independent set, so it cannot occur that recolouring some bad neighbour freezes another bad neighbour.
    Next we recolour $w$, if possible, which unfreezes all vertices in $B_2$.
    (In particular this is possible if $|B_2|\ge 2$, because fewer than $\deg(w)$ distinct colours are used by $\a$ on $N(w)$.)
    Now we recolour each bad neighbour in $B_2$.
    Finally, we recolour $w$ with $\b(w)$.
    This process succeeds in recolouring $w$ with $\b(w)$ unless $w$ has a single bad neighbour $x\in B_2$, and both $w$ and $x$ are frozen.
    Note that each neighbour of $w$ is recoloured at most once, and $w$ is recoloured at most twice (and twice only if $\c(w)=\b(w)$), so $|\mathcal{S}|\leq \deg(w)+2$, and the lemma holds.
\end{proof}

\begin{lem}
\label{fix-leaf-lem}
    Let $G$ be a connected graph, and let $L$ be a \pla{1} for $G$ such that $|L(v)|\geq \deg(v)+2$ for some vertex $v$.
    Let $\a, \b$ be two unfrozen $L$-colourings.
    Fix $w\in V(G)$, and let $P$ be a shortest $v,w$-path.
    There exists a recolouring sequence $\mathcal{S}$ from $\a$ to some $L$-colouring $\c$ such that
    (i) $\c(w)=\b(w)$,
    (ii) $v$ is unfrozen under $\c$,
    (iii) every vertex that is recoloured during $\mathcal{S}$ belongs to $V(P) \cup (N(w)\cap \a^{-1}(\b(w))$, and
    (iv) $|\mathcal{S}| \leq  \deg(w)+2+2|V(P)|$.
\end{lem}

\begin{proof}
    We first apply \Cref{lem:recolouringToUniqueBadNeighbour}, transforming $\a$ to an $L$-colouring $\c$.
    If $\c(w)=\b(w)$, then we are done.
    Otherwise, we assume that $w$ has a unique bad neighbour $x$, and that $w$ and $x$ are both frozen.
    Our general plan is to recolour along $P$ to unfreeze $w$ and then recolour $w$ and $x$ as desired; but the details are a bit more complicated. 
    Let $w'$ be the neighbour of $w$ on $P$.

    First suppose that $x$ lies on $P$; that is, $x$ is the penultimate vertex of $P$.
    We recolour from~$w$ along $P$ to unfreeze $x$, recolour $x$, and recolour $w$ with $\b(w)$.
    So assume $x$ is not on $P$ (and cannot be substituted for the penultimate vertex of $P$ to get another shortest $v,w$-path).
    Instead suppose that $x$ is adjacent to $w'$.
    Again, we recolour from $w$ along $P$ to unfreeze $x$, by recolouring $w'$.
    Afterwards, we recolour $x$ and recolour $w$ with $\b(w)$.

    Finally, assume that $x$ has no neighbours on $P$, except for $w$.
    We recolour from $v$ along $P$ to unfreeze $w$, recolour $w$, and recolour $x$.
    If $\b(w)$ is now available for $w$, then we use it there and are done.
    But possibly $\b(w)$ is now used on $w'$, the neighbour of $w$ on $P$; so assume this is the case.
    If~$w'$ is unfrozen, then recolour $w'$ and afterward recolour $w$ with $\b(w)$. 
    Otherwise, first recolour from $v$ along $P$ to unfreeze $w'$; thereafter, we finish as in the previous sentence.
    It is straightforward to check that (iv) holds.
\end{proof}

\begin{cor}
    The Key Lemma, precise version, holds.
\end{cor}

\begin{proof}
    Fix a connected graph $G$, a \pla{1} $L$ such that $|L(v)|\ge \deg(v)+2$ for some~$v$, and unfrozen $L$-colourings $\a,\b$.
    To recolour $\a$ to $\b$, we use induction on $|V(G)|$.
    Let $T$ be a spanning tree rooted at $v$, with a leaf $w$.
    By \Cref{fix-leaf-lem}, we can recolour $\a$ to an $L$-colouring $\c$ with $\c(w)=\b(w)$.
    Let $G':=G-w$, $L'(x):=L(x)\setminus\{\b(w)\}$ for all $x\in N(w)$, and $L'(x):=L(x)$ otherwise.
    Let $\a'$ and $\b'$ denote the restrictions to $G'$ of $\a$ and $\b$.
    By hypothesis, we can recolour~$\a'$ to~$\b'$.
    Along with the steps above, this recolours $\a$ to $\b$.

    Now, in addition, let $T$ be a breadth first tree rooted at $v$.
    Let $w_1,\ldots,w_n$ denote the vertices of $G$ in order of non-increasing distance from $v$, with $w_n=v$.
    Let $G_i$ be formed from $G$ by deleting $w_1,\ldots,w_{i-1}$, and let $\mathcal{S}_i$ denote the sequence in $G_i$ to recolour $w_i$ with $\b(w_i)$.
    Let $\mathcal{S}:=\mathcal{S}_1\cdots\mathcal{S}_n$, a concatenation of sequences.
    Now $|\mathcal{S}|=\sum_{i=1}^n|\mathcal{S}_i|\le \sum_{i=1}^n(\deg_{G_i}(w_i)+2+2|V(P_i)|) \le |E(G)|+2n +2\sum_{i=1}^ni = |E(G)|+2n+n(n+1)\le \tfrac12(3n^2+5n)$.
\end{proof}

Using a more refined version of the arguments in this section, we can prove the following quantitatively stronger version of the Key Lemma (when the minimum degree is at least $3$).

\begin{lem}
\label{lem:degreeplus_linearbound}
    There exists a constant $C$ such that the following holds.
    Let $G$ be a connected graph with minimum degree at least $3$ and with average degree~$\overline{d}$.
    If $L$ is a \pla{1} for~$G$ such that $|L(x)|\ge \deg(x)+2$ for some $x\in V(G)$, then $\diam \CLG\le C\overline{d}|V(G)|$.
\end{lem}

Here we sketch the proof of~\cref{lem:degreeplus_linearbound}, which goes in three stages.
The first stage is a technical variant of~\cref{fix-leaf-lem} that allows $v$ to be \emph{any} unfrozen vertex of degree at least $3$, but in return requires $w$ to be at distance $5$ or more from $v$.
Starting from a colouring with $v$ unfrozen, this stage produces a colouring in which $v$ is still unfrozen and $w$ gets colour $\b(w)$.
It requires at most $2|P|+ \deg(w)+20$ recolourings of vertices in  $P \cup N(w)\cup N(v)$, where $P$ is a shortest $v,w$-path.

In the second stage we simultaneously unfreeze all vertices of a maximal set of vertices $X$ that are pairwise at distance at least $10$.
We also require that $X$ includes the special vertex $x$ with $|L(x)|\geq \deg(x)+2$.
Note that every vertex must be at distance no more than $10$ from $X$.

In the third stage we iteratively apply the technical variant to a vertex $v \in X$ and some other vertex $w$ at some distance $d$ from $v$, where $d$ is small (say $d<30$) but not too small ($d\ge 5$).
The technical variant ensures that $w$ is recoloured to its target colour while keeping $X\setminus\{w\}$ unfrozen;
it requires at most $2 \cdot 30 + \deg(w) + 20$ recolouring steps.
We can then finish by induction on $G-w$.
Note however that we must carefully choose the order in which the vertices $w$ are removed; otherwise,  removing some $w$ from $G$ could disconnect the graph or create a vertex whose distance to $X$ is too large.
At the end, only vertices at small distance (less than $5$) from the special vertex $x$ remain.
So we can finish by iteratively applying~\cref{fix-leaf-lem} itself, recolouring vertices working inward from the leaves of a breadth first tree rooted at $u$.
This argument in fact yields $\diam \CLG\le (c+ \tfrac12\overline{d}) |V(G)|$, for some constant $c>0$.

We have omitted the details of this proof, since it is technical and we have not proved that even $\diam \CLG=o(|V(G)|^2)$.
Nonetheless, it is natural to ask whether one can prove a still stronger bound in the Key Lemma: $\diam\CLG\le C|V(G)|$ when the minimum degree is at least~$3$, removing the dependence on $\overline{d}$.

\section{Regular Graphs with Identical Lists}
\label{regular-sec}

In this section, we prove the Main Theorem in the special case that $G$ is a $2$-connected regular graph and all lists are identical. 
This implies the main result of Feghali et al.~\cite{FJP-brooks} in the case when~$G$ is $2$-connected.
It is also a crucial step in our proof of the Main Theorem.

\begin{remark}
Henceforth, for brevity whenever a graph $G$ has a \pla{1} $L$ we will assume that $|L(v)|=\deg(v)+1$ for all $v\in V(G)$.
Otherwise, we are done by the Key Lemma.
\end{remark}

\begin{defi}
For a graph $G$, a list-assignment $L$ for $V(G)$, and $L$-colourings $\a$ and $\b$, we write $\a\sim_L\b$ to denote that we can recolour $\a$ to $\b$ by a sequence of single vertex recolourings, at each step maintaining an $L$-colouring.
When $L$ is clear from context, we will often simply write $\a\sim\b$.
\end{defi}

\begin{obs}
\label{move-unfrozen-obs}
    Fix a connected graph $G$ and a \pla{1} $L$.
    To show that $\a\sim \b$ for
    given unfrozen $L$-colourings $\a$ and $\b$, it suffices to specify arbitrary vertices $v,w$ and handle the case that $v$ is unfrozen in
    $\a$ and $w$ is unfrozen in $\b$.
\end{obs}

\begin{proof}
    Since $\a$ is unfrozen, it has some vertex $z$ that is unfrozen; pick such a $z$ that is nearest to~$v$.
    Now recolour along a shortest path from $z$ to $v$ to unfreeze $v$; call the resulting $L$-colouring $\a'$.
    Clearly $\a\sim \a'$.
    Similarly, there exists an $L$-colouring $\b'$ with $\b'\sim \b$ and in which $w$ is unfrozen.
    If we can show that $\a'\sim\b'$, then we have $\a\sim \a' \sim \b'\sim\b$, as desired.
\end{proof}

When aiming to prove that $\a\sim\b$, a helpful idea is to recolour $\a$ to use a common colour $c$ on two neighbours $w_1,w_2$ of some vertex $v$, and let $G':=G-\{w_1,w_2\}$.
If $G'$ is connected, then we can recolour it arbitrarily (except for using $c$ on $N(w_1)\cup N(w_2)$), by applying the Key Lemma to~$v$ in~$G'$. 
This approach motivates our next set of definitions.

\begin{defi}
\label{good-pair-defn}
    Let $G$ be a connected graph, $\a$ a (list) colouring of $G$, and $v\in V(G)$.
    A \emph{good pair} for $v$ w.r.t.~$\a$ is a pair of vertices $w_1,w_2\in N(v)$ such that $\a(w_1)=\a(w_2)$; in particular, $w_1w_2\notin E(G)$.
    A \emph{very good pair} is a good pair such that $G-\{w_1,w_2\}$  is connected. 
\end{defi}

If $\a$ is a $(\Delta+1)$-colouring, $\deg(v)=\Delta$, and $v$ is unfrozen in $\a$, then the $\Delta$ neighbours of $v$ are coloured with at most $\Delta-1$ colours.
Thus, by the pigeonhole principle, some pair of neighbours of $v$ use the same colour, so are a good pair for $v$ w.r.t.~$\a$.
In particular, if $G$ is $\Delta$-regular, then every unfrozen vertex has a good pair.
If $G$ is $3$-connected, then every good pair is in fact a very good pair.

A partial (proper) $L$-colouring assigns colours to some vertices $v$ (from $L(v)$), but may leave some vertices uncoloured.
To \emph{extend} a partial $L$-colouring $\c_0$ to an $L$-colouring $\c$ means to construct a proper $L$-colouring $\c$ such that $\c(v)=\c_0(v)$ for all vertices $v$ where $\c_0(v)$ is defined.
If $|L(v)|\ge \deg(v)+1$ for every $v\in V(G)$, then every partial $L$-colouring $\c_0$ can be extended to an $L$-colouring~$\c$.
In particular, this is true when $L(v)=\{1,\ldots,\Delta+1\}$ for all $v$.
We call this \emph{extending} $\c_0$ to $\c$.

\begin{lem}
\label{twomatch-lem}
    Let $G$ be a connected graph with a \pla{1} $L$.
    Let $\a,\b$ be $L$-colourings of~$G$.
    If there exist $v\in V(G)$ and $w_1,w_2\in N(v)$ such that $\a(w_1)=\b(w_1)$ and $w_1,w_2$ are a very good pair for $v$ w.r.t.\ both $\a$ and $\b$, then $\a \sim \b$.
    Furthermore, there exists a recolouring sequence from $\a$ to~$\b$ that recolours neither $w_1$ nor $w_2$.
\end{lem}

\begin{proof}
    Let $G':=G-\{w_1,w_2\}$, let $L'(x):=L(x)\setminus \{\a(w_1)\}$ for all $x\in N(w_1)\cup N(w_2)$, and otherwise let $L'(x):=L(x)$.
    We are done by applying the Key Lemma to $G'$. 
\end{proof}

\begin{lem}
\label{3connected-reg-lem}
    Let $G$ be a $3$-connected $\Delta$-regular graph.
    If $\a$ and $\b$ are unfrozen $(\Delta+1)$-colourings of~$G$, then $\a\sim\b$.
\end{lem}

\begin{proof}
    Let $v$ be an arbitrary unfrozen vertex in $\a$.
    By \Cref{good-pair-defn}, vertex $v$ has a very good pair $w_1,w_2$ for $\a$.
    We may assume that $w_1$ is unfrozen in $\b$.
    Thus, $w_1$ has a very good pair $x_1,x_2$.
    Note that $\{w_1,w_2\}\cap \{x_1,x_2\}=\varnothing$, since $x_1,x_2\in N(w_1)$, but $w_2\notin N(w_1)$.
    First suppose that $\a(w_1)\ne \b(x_1)$.
    Form a proper $(\Delta+1)$-colouring $\c$ by colouring $w_1,w_2$ with $\a(w_1)$, colouring $x_1,x_2$ with~$\b(x_1)$, and extending.
    By applying \Cref{twomatch-lem} twice, we get $\a \sim \c \sim \b$.
    So assume instead, without loss of generality, that $\a(w_1)=\b(x_1)=1$.
    For each ordered pair $(i,j)\in \{(1,2),(3,2),(3,1)\}$, form a proper $(\Delta+1)$-colouring $\c_{i,j}$ by colouring $w_1,w_2$ with $i$, by colouring $x_1,x_2$ with $j$, and extending.
    By applying \Cref{twomatch-lem} four times, we get $\a \sim \c_{1,2} \sim \c_{3,2} \sim \c_{3,1} \sim \b$; see \Cref{lem11fig}.
\end{proof}

\begin{figure}[!ht]
\centering
\begin{tikzpicture}[xscale=.75, scale=.8, yscale=-1]
    \draw (-6.5,0) node[lStyle] {~};
    \def\myeps{.35cm}
    \begin{scope}[xshift=-6.25 cm] 
        \filldraw[gray!15!white,rounded corners=.15in] (0,-1) rectangle (3,2);
        \foreach \i in {1,2} 
        {
            \draw (\i,0) node (v\i0) {};
            \draw (\i,1) node (v\i1) {};
            \draw[dashed] (v\i0) -- (v\i1);
        }
        \draw (v10) ++ (0,-\myeps) node[lStyle] {\footnotesize{$w_1$}};
        \draw (v11) ++ (0,\myeps) node[lStyle] {\footnotesize{$w_2$}};
        \draw (v20) ++ (0,-\myeps) node[lStyle] {\footnotesize{$x_1$}};
        \draw (v21) ++ (0,\myeps) node[lStyle] {\footnotesize{$x_2$}};
    \end{scope}

    \def\mypush{4.25}
    \foreach \a/\b/\x in {1/{}/0, 1/2/1, 3/2/2, 3/1/3, {}/1/4}
    {
    \begin{scope}[xshift=\x*\mypush cm]
        \filldraw[gray!15!white,rounded corners=.15in] (0,-1) rectangle (3,2);
        \foreach \i in {1,2} 
        {
            \draw (\i,0) node (v\i0) {};
            \draw (\i,1) node (v\i1) {};
            \draw[dashed] (v\i0) -- (v\i1);
        }
        \draw (v10) ++ (0,-1.1*\myeps) node[lStyle]{\footnotesize{\a}};
        \draw (v11) ++ (0,1.1*\myeps) node[lStyle] {\footnotesize{\a}};
        \draw (v20) ++ (0,-1.1*\myeps) node[lStyle] {\footnotesize{\b}};
        \draw (v21) ++ (0,1.1*\myeps) node[lStyle] {\footnotesize{\b}};
    \end{scope}

    \foreach \x in {.85, 1.85, 2.85, 3.85}
    \draw (\x*\mypush cm,.5) node[lStyle] {$\sim$};

    \foreach \x/\lab in {0/{\alpha}, 1/{\c_{1,2}}, 2/{\c_{3,2}}, 3/{\c_{3,1}}, 4/{\beta}}
    \draw (1.5cm+\x*\mypush cm,2.5) node[lStyle] {$\lab$};
    }
\end{tikzpicture}
    \captionsetup{width=.895\textwidth}
  \caption{The harder case in the proof of \Cref{3connected-reg-lem}, and also of Claim~1 in the proof of \Cref{regular-thm}:
  Vertex names (left) and key colourings (right) in the reconfiguration sequence from $\a$ to $\b$.\label{lem11fig}}
\end{figure}
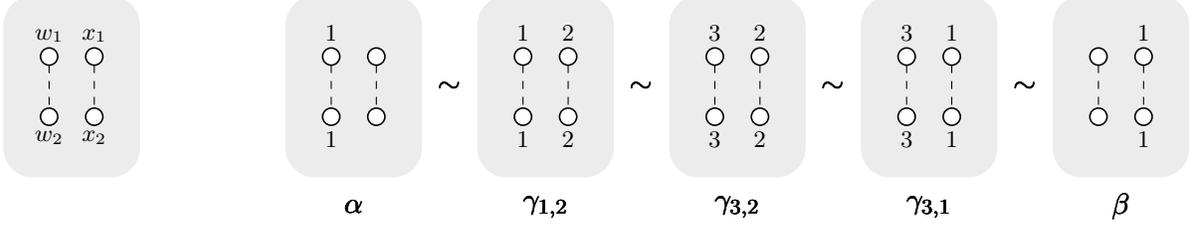

In fact, we can weaken the hypothesis in the previous lemma that $G$ is $3$-connected, to require only that $G$ is $2$-connected.
Our general idea in the proof below is the same as that in the proof above. 
However, now it is no longer the case that every good pair is in fact a very good pair.
So we must expend more effort ensuring that the pairs we pick are very good.

\begin{theorem}
\label{regular-thm}
    Let $G$ be a $2$-connected $\Delta$-regular graph, with $\Delta\ge 3$.
    If $\a$ and $\b$ are unfrozen $(\Delta+1)$-colourings of $G$, then $\a\sim\b$.
\end{theorem}

\begin{proof}
    If $G$ is $3$-connected, then we are done by the previous lemma.
    So instead assume that~$G$ has a $2$-cut.
    Among all $2$-cuts in $G$, pick a $2$-cut $S=\{z_1,z_2\}$ that leaves a biggest possible component~$H_1$.
    Let~$H_2$ be another component of $G-S$.
    
    \setcounter{claim}{0}
    \begin{claim}
    \label{clm1A}
    Fix $v_1,v_2\in V(G)$ such that $v_1$ is unfrozen in $\a$ and that $v_2$ is unfrozen in $\b$.
    If there exists a very good pair $x_1,x_2$ for $v_1$ w.r.t.~$\a$, there exists a very good pair $y_1,y_2$ for $v_2$ w.r.t.~$\b$, and $\{x_1,x_2\}\cap \{y_1,y_2\}= \varnothing$, then the theorem holds.
    \end{claim}

    \begin{claimproof}
    The proof is the same as that of~\Cref{3connected-reg-lem}.
    First suppose that $\a(x_1)\ne \b(y_1)$.
    Form a proper $(\Delta+1)$-colouring $\c$ by colouring $x_1,x_2$ with $\a(x_1)$, colouring $y_1,y_2$ with $\b(y_1)$, and extending.
    By applying \Cref{twomatch-lem} twice, we have $\a \sim \c \sim \b$.
        
    So assume instead, without loss of generality, that $\a(w_1)=\b(x_1)=1$.
    For each ordered pair $(i,j)\in \{(1,2),(3,2),(3,1)\}$, form a proper $(\Delta+1)$-colouring $\c_{i,j}$ by colouring $x_1,x_2$ with $i$, by colouring $y_1,y_2$ with $j$, and extending.
    By applying \Cref{twomatch-lem} four times, we have $\a \sim \c_{1,2} \sim \c_{3,2} \sim \c_{3,1} \sim \b$; again, see \Cref{lem11fig}.
    \end{claimproof}

    \begin{claim}
    \label{clm2A}
    There exists a vertex $v\in V(H_2)$ such that (i) $S\not\subseteq N(v)$ and (ii) if $\c$ is a $(\Delta+1)$-colouring of $G$ and $v$ is unfrozen in $\c$, then $v$ has a good pair w.r.t.~$\c$, and every such good pair is a very good pair.
    In fact, this property holds for every $v\in V(H_2)$ such that $S\not\subseteq N(v)$.
    \end{claim}

    \begin{claimproof}
    We prove the second statement first.
    Suppose there exists $v\in V(H_2)$ such that $S\not\subseteq N(v)$; see the left of \Cref{clm2-fig}.
    Now consider a $(\Delta+1)$-colouring $\c$ of $G$ in which $v$ is unfrozen.
    By \Cref{good-pair-defn}, there exist $w_1,w_2\in N(v)$ such that $\c(w_1)=\c(w_2)$.
    If $G-\{w_1,w_2\}$ is not connected, then all vertices of $H_1$ lie in the same component; furthermore, this component also contains at least one vertex of $S$, since $S\not\subseteq N(v)$.
    But this is a larger component than $H_1$, contradicting our choice of $S$ and $H_1$.
    Thus, $G-\{w_1,w_2\}$ is connected, as desired.
    Since $\c$ was arbitrary, we have found the desired vertex $v$; this proves the second statement.

    Now we prove the first statement.
    Assume instead that each $v\in V(H_2)$ has $S\subseteq N(v)$; see the right of \Cref{clm2-fig}.
    For each $z_i\in S$, we have $\deg_{H_2}(z_i)=|V(H_2)|$. 
    And for each $v\in V(H_2)$, we have $\deg_G(v)\le |V(H_2)|-1+|S| = |V(H_2)|+1$. 
    Since $G$ is regular, each $z_i\in S$ has exactly one neighbour outside~$H_2$; in particular $H_1$ and $H_2$ are the only components of $G-S$, since $G$ is $2$-connected and thus each component of $G-S$ must be adjacent to each vertex of $S$.
    Furthermore, $H_2$ is complete and each vertex of $S$ is complete to $H_2$.
    If $\c$ is a $(\Delta+1)$-colouring of $G$ with $z_i$ unfrozen, then $z_i$ has a good pair $y_1,y_2$ w.r.t.~$\c$.
    Since $H_2$ is complete, after possibly renaming, we have $y_1\in V(H_1)$ and $y_2\in V(H_2)$.
    Now $G-y_1$ is connected since $G$ is $2$-connected.
    And $G-\{y_1,y_2\}$ is connected because $N(y_2)$ is a clique minus an edge (with endpoints in $S$), so any path in $G-y_1$ that passes through $y_2$ can be rerouted in $G-\{y_1,y_2\}$ through some neighbour of~$y_2$ in $H_2$.

    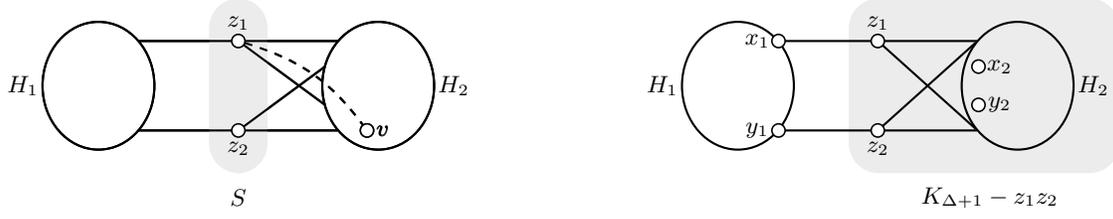
\begin{figure}[!ht]
    \centering
\begin{tikzpicture}[thick, scale=.85, xscale=.875]
\tikzstyle{uStyle}=[shape = circle, minimum size = 5.0pt, inner sep = 0pt,
outer sep = 0pt, draw, fill=white, semithick]
\tikzstyle{sStyle}=[shape = rectangle, minimum size = 4.5pt, inner sep = 0pt,
outer sep = 0pt, draw, fill=white, semithick]
\tikzstyle{lStyle}=[shape = circle, minimum size = 4.5pt, inner sep = 0pt,
outer sep = 0pt, draw=none, fill=none]
\tikzset{every node/.style=uStyle}
\clip (-2,-2.05) rectangle (19,1.5);

    \def\myeps{2.8mm}

    \filldraw[gray!15!white,rounded corners=.15in] (2,-1.35) rectangle (3,1.35);
    \draw (0,0) circle (1cm) (5,0) circle (1cm);
    \draw (2.5,.7) node (z1) {} (2.5,-.7) node (z2) {};
    \draw (z1) --++ (-1.775,0) (z1) --++ (1.775,0);
    \draw (z2) --++ (-1.775,0) (z2) --++ (1.775,0);
    \draw (z1) --++ (1.55,-1) (z2) --++ (1.55,1);
    \draw (z1) ++ (0,\myeps) node[lStyle] {\scriptsize{$z_1$}};
    \draw (z2) ++ (0,-\myeps) node[lStyle] {\scriptsize{$z_2$}};
    \draw (z2) ++ (2.3,0) node (v) {};
    \draw (v) ++ (\myeps,0) node[lStyle] {\scriptsize{$v$}};
    \draw[dashed] (z1) to [out =-15 , in = 130] (v);

    \filldraw[gray!15!white,rounded corners=.15in] (2,-1.35) rectangle (3,1.35);
    \draw (0,0) circle (1cm) (5,0) circle (1cm);
    \draw (2.5,.7) node (z1) {} (2.5,-.7) node (z2) {};
    \draw (z1) --++ (-1.775,0) (z1) --++ (1.775,0);
    \draw (z2) --++ (-1.775,0) (z2) --++ (1.775,0);
    \draw (z1) --++ (1.55,-1) (z2) --++ (1.55,1);
    \draw (z1) ++ (0,\myeps) node[lStyle] {\footnotesize{$z_1$}};
    \draw (z2) ++ (0,-\myeps) node[lStyle] {\footnotesize{$z_2$}};
    \draw (z2) ++ (2.3,0) node (v) {};
    \draw (v) ++ (\myeps,0) node[lStyle] {\footnotesize{$v$}};
    \draw[dashed] (z1) to [out =-15 , in = 130] (v);
    \draw (-1.35,0) node[lStyle] {\footnotesize{$H_1$}};
    \draw (6.35,0) node[lStyle] {\footnotesize{$H_2$}};
    \draw (2.5,-1.75) node[lStyle] {\footnotesize{$S$}};

    \begin{scope}[xshift=4.5in]
        \filldraw[gray!15!white,rounded corners=.15in] (2,-1.35) rectangle (6.8,1.35);
        \draw (4.5,-1.75) node[lStyle] {\footnotesize{$K_{\Delta+1}-z_1z_2$}};
        \draw (0,0) circle (1cm) (5,0) circle (1cm);
        \draw (2.5,.7) node (z1) {} (2.5,-.7) node (z2) {};
        \draw (z1) --++ (-1.775,0) node (x1) {} ++ (-1.35*\myeps,0) node[lStyle] {\footnotesize{$x_1$}} (z1) --++ (1.775,0);
        \draw (z2) --++ (-1.775,0) node (y1) {} ++ (-1.35*\myeps,0) node[lStyle] {\footnotesize{$y_1$}} (z2) --++ (1.775,0);
        \draw (z1) --++ (1.75,-1.375) (z2) --++ (1.75,1.375);
        \draw (z1) ++ (0,\myeps) node[lStyle] {\footnotesize{$z_1$}};
        \draw (z2) ++ (0,-\myeps) node[lStyle] {\footnotesize{$z_2$}};
        \draw (z1) ++ (1.8,-.4) node (x2) {} ++ (1.35*\myeps,0) node[lStyle] {\footnotesize{$x_2$}};
        \draw (z2) ++ (1.8,.4) node (y2) {} ++ (1.35*\myeps,0) node[lStyle] {\footnotesize{$y_2$}};
        \draw (v) ++ (\myeps,0) node[lStyle] {\footnotesize{$v$}};
        \draw (-1.35,0) node[lStyle] {\footnotesize{$H_1$}};
        \draw (6.35,0) node[lStyle] {\footnotesize{$H_2$}};
    \end{scope}
\end{tikzpicture}
\captionsetup{width=.85\textwidth}
 \caption{Left: The general case, guaranteed by Claim 2, when there exists $v\in V(H_2)$ with $S\not\subseteq N(v)$.
    Right: An exceptional case, when no such $v$ exists, and $G[V(H_1)\cup S]=K_{\Delta+1}-z_1z_2$.\label{clm2-fig}}
\end{figure}

    We assume that $z_1$ is unfrozen in $\a$, and that $z_2$ is unfrozen in $\b$.
    Let $x_1,x_2$ be a good pair for~$z_1$ w.r.t.~$\a$, and let $y_1,y_2$ be a good pair for $z_2$ w.r.t.~$\b$.
    By the previous paragraph, each of these good pairs is in fact very good.
    By possibly swapping indices, we assume that $x_1,y_1\in V(H_1)$ and $x_2,y_2\in V(H_2)$.
    Since $z_1,z_2$ each have only a single neighbour in $V(H_1)$, and $G$ is 2-connected and $\Delta$-regular for $\Delta\geq 3$, we know that $x_1\ne y_1$.
    If $x_2\ne y_2$, then the theorem holds by \Cref{clm1A}, and we are done; so assume instead that $x_2=y_2$.
    By symmetry, we assume that colours 1 and 2 are both available for $z_2$ in $\b$, and that $\b(y_1)=\b(y_2)=3$.
    If $y_2$ is unfrozen, then we recolour it, and reuse colour 3 on another neighbour of $z_2$ in $H_2$.
    So assume instead that $y_2$ is frozen.
    Now in $\b$ recolour $z_2$ to unfreeze $y_2$, then proceed as above.
    Now redefining the good pairs $x_1,x_2$ and $y_1,y_2$ as at the start of the paragraph gives $\{x_1,x_2\}\cap \{y_1,y_2\}=\varnothing$, so the theorem holds by \Cref{clm1A} as explained above.
    This proves the first statement.
    \end{claimproof}

    \begin{claim}
    \label{clm3A}
        There exists $z\in N(z_1)\cap N(z_2)\cap V(H_2)$ and $V(H_2)\subseteq N(z_1)\cup N(z_2)$.
    \end{claim}

    \begin{claimproof}
    By \Cref{clm2A}, there exists $v\in V(H_2)$ such that $S\not\subseteq N(v)$; again, see the left of \Cref{clm2-fig}.
    We assume that $v$ is unfrozen in~$\a$.
    By \Cref{good-pair-defn}, we know $v$ has a good pair $w_1,w_2$ w.r.t.~$\a$.
    By symmetry, we assume that $w_1\notin S$.
    If $S\not\subseteq N(w_1)$, then we assume that $w_1$ is unfrozen in $\b$.
    Similarly, $w_1$ has a good pair $x_1,x_2$ w.r.t.~$\b$.
    Note that $\{w_1,w_2\}\cap \{x_1,x_2\}=\varnothing$, so we are done by \Cref{clm1A}.
    Thus, we assume that $S\subseteq N(w_1)$.
    Let $z:=w_1$.
    This proves the first statement.

    Now we prove the second statement.
    Suppose, to the contrary, that there exists $w\in V(H_2)\setminus (N(z_1)\cup N(z_2))$.
    We assume that $w$ is unfrozen in both $\a$ and $\b$.
    Thus, $w$ has a good pair $x_1,x_2$ w.r.t.~$\a$, and $w$ has a good pair $y_1,y_2$ w.r.t.~$\b$.
    Note that each of these good pairs is very good by~\cref{clm2A}.
    Fix a colour $c\notin\{\a(x_1),\b(y_1)\}$.
    Form a proper $(\Delta+1)$-colouring $\c_1$ by colouring $x_1,x_2$ with $\a(x_1)$, by colouring $z_1,z_2$ with $c$, and by extending.
    Similarly, form a proper $(\Delta+1)$-colouring~$\c_2$ by colouring $y_1,y_2$ with $\b(y_1)$, by colouring $z_1,z_2$ with $c$, and by extending.
    Furthermore, we can choose $\c_2$ to agree with $\c_1$ on all components of $G-\{z_1,z_2\}$ other than $H_2$.
    Now applying \Cref{twomatch-lem} three times, as we explain below, gives $\a\sim \c_1\sim \c_2 \sim \b$; see \Cref{clm3fig}.

    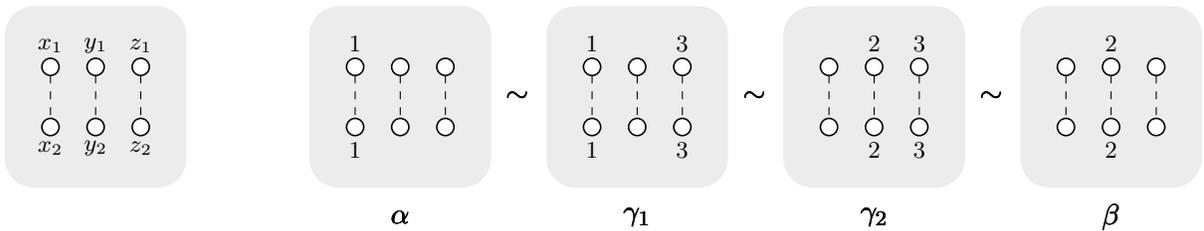
\begin{figure}[!ht]
    \centering
\begin{tikzpicture}[xscale=.75, scale=.8, yscale=-1]
    \draw (-7,0) node[lStyle] {~};
    \def\myeps{.35cm}
    \begin{scope}[xshift=-6.75 cm] 
        \filldraw[gray!15!white,rounded corners=.15in] (0,-1) rectangle (4,2);
        \foreach \i in {1,2,3} 
        {
            \draw (\i,0) node (v\i0) {};
            \draw (\i,1) node (v\i1) {};
            \draw[dashed] (v\i0) -- (v\i1);
        }
        \draw (v10) ++ (0,-\myeps) node[lStyle] {\footnotesize{$x_1$}};
        \draw (v11) ++ (0,\myeps) node[lStyle] {\footnotesize{$x_2$}};
        \draw (v20) ++ (0,-\myeps) node[lStyle] {\footnotesize{$y_1$}};
        \draw (v21) ++ (0,\myeps) node[lStyle] {\footnotesize{$y_2$}};
        \draw (v30) ++ (0,-\myeps) node[lStyle] {\footnotesize{$z_1$}};
        \draw (v31) ++ (0,\myeps) node[lStyle] {\footnotesize{$z_2$}};
    \end{scope}

    \def\mypush{5.25}
    \foreach \a/\b/\c/\x in {1/{}/{}/0, 1/{}/3/1, {}/2/3/2, {}/2/{}/3}
    {
    \begin{scope}[xshift=\x*\mypush cm]
        \filldraw[gray!15!white,rounded corners=.15in] (0,-1) rectangle (4,2);
        \foreach \i in {1,2,3} 
        {
            \draw (\i,0) node (v\i0) {};
            \draw (\i,1) node (v\i1) {};
            \draw[dashed] (v\i0) -- (v\i1);
        }
        \draw (v10) ++ (0,-1.1*\myeps) node[lStyle]{\footnotesize{\a}};
        \draw (v11) ++ (0, 1.1*\myeps) node[lStyle] {\footnotesize{\a}};
        \draw (v20) ++ (0,-1.1*\myeps) node[lStyle] {\footnotesize{\b}};
        \draw (v21) ++ (0, 1.1*\myeps) node[lStyle] {\footnotesize{\b}};
        \draw (v30) ++ (0,-1.1*\myeps) node[lStyle] {\footnotesize{\c}};
        \draw (v31) ++ (0, 1.1*\myeps) node[lStyle] {\footnotesize{\c}};
    \end{scope}

    \foreach \x in {.875, 1.875, 2.875}
    \draw (\x*\mypush cm,.5) node[lStyle] {$\sim$};

    \foreach \x/\lab in {0/{\alpha}, 1/{\gamma_1}, 2/{\gamma_2}, 3/{\beta}}
    \draw (2cm+\x*\mypush cm,2.5) node[lStyle] {$\lab$};
    }
\end{tikzpicture}
\captionsetup{width=.895\textwidth}
    \caption{Proving Claim 3: Vertex names (left) and key colourings in the reconfiguration sequence (right). 
    For simplicity, we show the case when $\a(x_1)=1, \b(y_1)=2, c=3$, and $\{x_1,x_2\}\cap \{y_1,y_2\}=\varnothing$.
    However, this intersection could be non-empty and/or we might have $\b(y_1)=\a(x_1)$.\label{clm3fig}}
\end{figure}

    Here $\a \sim \c_1$ and $\c_2 \sim \b$ follow by applying \Cref{twomatch-lem} to the very good pairs $x_1,x_2$ and $y_1,y_2$ in $G$, respectively.
    For $\c_1 \sim \c_2$ the argument is a bit more subtle since $z_1,z_2$ is not a very good pair in $G$.
    Instead we let $G':=G[V(H_2 \cup \{z_1,z_2\}]$, and use that $z_1,z_2$ is a very good pair in $G'$ with respect to the restrictions $\c_1{'},\c_2{'}$ of $\c_1,\c_2$ to $G{'}$.
    By \Cref{twomatch-lem}, we not only have that $\c_1{'} \sim \c_2{'}$ but also during the implied reconfiguration the colours of the cut-set $\{z_1,z_2\}$ never change.
    Since~$\c_1$ and $\c_2$ agree on $G-V(G{'})$, it follows that also $\c_1 \sim \c_2$.
    This proves the second statement.
    \end{claimproof}

    Consider $v\in V(H_2)$ such that $S\not\subseteq N(v)$; such a $v$ is guaranteed by \Cref{clm2A}.
    We assume that~$v$ is unfrozen in both $\a$ and $\b$.
    So by \Cref{good-pair-defn}, vertex $v$ has a good pair w.r.t.~$\a$ and also has a good pair w.r.t.~$\b$.
    The argument proving the second statement of \Cref{clm3A} actually works unless one of these good pairs intersects $S$.
    So assume that the good pair for $v$ w.r.t.~$\a$ is $w_1,w_2$ with $w_1\in S$; by symmetry, say $w_1=z_1$.
    This implies that $w_2\notin N(z_1)$.
    So assume that $\b$ is unfrozen for $w_2$.
    Now $w_2$ has a good pair $x_1,x_2$ w.r.t.~$\b$.
    Note that $\{w_1,w_2\}\cap \{x_1,x_2\}=\varnothing$.
    Furthermore, both of these good pairs are very good.
    Thus, we are finished by \Cref{clm1A}.
\end{proof}

\section{Proof of Connectedness}
\label{connect-sec}

In this section, we prove our Main Theorem:
Let $G$ be a connected graph with $\Delta\ge 3$ and let $L$ be a \pla{1} list-assignment of $G$.
If $\a,\b$ are unfrozen $L$-colourings, then $\a\sim \b$.
To begin, we sketch the proof.

If $G$ is $2$-connected, then we consider whether there exist neighbours $v,w$ such that $L(v)\not\subseteq L(w)$.
If not, then $G$ is regular with identical lists, and so we are done by \Cref{regular-thm}.
If so, then we recolour~$\a$ to get an $L$-colouring $\a'$ with $\a'(v)\notin L(w)$.
From $\a'$, we use the Key Lemma (applied to $G-v$) to reach an $L$-colouring $\b'$ that agrees with $\b$ except possibly on $N[v]$; and repairing the colouring on $N[v]$ takes a bit of work.
This is \Cref{2connected-lem}.

The bulk of our work goes toward the case when $G$ has a cut-vertex.
Our proof goes by induction on the number of blocks.
We find an endblock $H$ (block with a unique cut-vertex $v$), recolour $H-v$ from $\a$ to match $\b$, and recurse on $G-(H-v)$. 
One important subtlety is that our current colouring must remain unfrozen when restricted to this subgraph on which we recurse.
When $H$ is neither complete nor a cycle, we can either reuse \Cref{2connected-lem}, mentioned above, or we can find a very good pair in $H$, use the Key Lemma to recolour $G-H$ to agree with $\b$, and finish on $H$ by the $2$-connected case above.
Specifically, if $z_1,z_2\in V(H)-v$ is a very good pair for some vertex $w\in V(H)-v$, then we apply the Key Lemma to $G-\{z_1,z_2\}$, which allows us to recolour that subgraph arbitrarily, apart from avoiding on $N(z_1)\cup N(z_2)$ the colour that we are using on $z_1,z_2$.
But since $G-H\subseteq G-N(z_1)\cup N(z_2)$, we can recolour $G-H$ to match $\b$, as desired.

When $H$ is either complete or a cycle, we can easily reduce to the case that $H=K_2$.
Now walking away from the leaf $v$ in this $K_2$, we continue until we reach a vertex $w$ with $\deg(w)\ge 3$.
For this $v,w$-path $P$, we show how to get the desired colours on $P-w$, and again recurse on $G-(P-w)$. 
To do this, we consider the possibilities that $w$ has two neighbours off of $P$ that are either (i) adjacent or (ii) non-adjacent.
These final cases also serve as the base cases of our induction.

\subsection{Some Simple Cases}
\label{sec4.1}

In this subsection we handle some simple cases: cycles, paths, complete graphs, the claw, the paw, and $2$-connected graphs with not all lists identical.
These preliminary results both help develop intuition, as well as simplify the proof of the Main Theorem.
First we explain why in our Main Theorem we require that $\Delta\ge 3$.
If we did not, the statement would be false, as follows.

\begin{example}
    Let $C_n$ be a cycle with vertices $v_1,\ldots,v_n$, and let $e_i=v_iv_{i+1}$ (with subscripts modulo $n$).
    For a $3$-colouring $\a$,
    let $f_{\a}(e_i):=+1$ if $(\a(v_i),\a(v_{i+1}))\in \{(1,2),(2,3),(3,1)\}$ and $f_{\a}(e_i):=-1$ if $(\a(v_i),\a(v_{i+1}))\in \{(2,1),(3,2),(1,3)\}$.
    Finally, let $f(\a):=\frac13\sum_{i=1}^nf_{\a}(e_i)$.
    We call~$f(\a)$ the \emph{winding number} of $\a$.
    It is straightforward to check that $f(\a)$ is an integer for all $\a$.

    It is also easy to verify that if $\a\sim \b$, then $f(\a)=f(\b)$.
    That is, $f$ is invariant under recolouring.
    To prove this invariance, consider a vertex $v_i$ before and after a recolouring step.
    If $v_i$ is recoloured from $j$ to $k$, then $v_{i+1}$ and $v_{i-1}$ are both coloured $\ell$, where $j,k,\ell\in\{1,2,3\}$ are distinct.
    Say $\a$ and~$\a'$ are the colourings before and after we recolour $v_i$.
    So $\a(e_i)=-\a(e_{i-1})$ and $\a'(e_i)=-\a'(e_{i-1})$.
    Thus, $f(\a')=f(\a)$, as desired.

    For $n\ge 8$, we construct $3$-colourings $\a$ and $\b$ of $C_n$ with $f(\a)\ne f(\b)$.
    For $\a$, colour $v_1,\ldots,v_8$ as $1,2,1,2,1,2,1,2$; and for $\b$ as $1,2,3,1,2,3,1,2$.
    If $n>8$, then we extend both precolourings arbitrarily, but identically.
    We have $\sum_{i=1}^7f_{\a}(e_i)=1$, but $\sum_{i=1}^7f_{\b}(e_i)=7$.
    Thus, $f(\a)=f(\b)-2$.
    In fact, the number of equivalence classes of $3$-colourings of $C_n$ is at least $\lfloor (n+4)/6\rfloor$.
\end{example}

To avoid the subtleties in the previous example, we rely on the following sufficient condition.

\begin{lem}
\label{cycle-lem}
    Let $G=C_n$, for some $n\ge 3$.
    Fix a $3$-assignment $L$ for $G$.
    If $\a$ and $\b$ are $L$-colourings, neither of which is frozen, and $L$ is not identical on all vertices, then $\a \sim \b$.
\end{lem}

\begin{proof}
    Denote the vertices of $G$ by $v_1,\ldots,v_n$.

\textbf{Case 1: There exists $\bm{i\in \{1,\ldots,n\}}$ such that $\bm{\b(v_i)\notin L(v_{i+1})}$, with subscripts modulo $\bm{n}$.}
    First we recolour vertices so that $v_i$ uses $\b(v_i)$; afterward, we can finish by applying the Key Lemma to $G-v_i$, since $|L(v_{i+1})|=3=\deg_{G-v_i}(v_{i+1})+2$.
    So assume that $\a(v_i)\ne \b(v_i)$, and further that $\a(v_{i-1})=\b(v_i)$; otherwise, we simply recolour $v_i$ with $\b(v_i)$.

    Since $\a$ is unfrozen, there exists $v_j$ that is unfrozen in $\a$; we assume $j=1$ (shifting indices if needed).
    We recolour $v_1$, if needed, to unfreeze $v_2$; recolour $v_2$, if needed, to unfreeze $v_3$, etc.
    Eventually, we unfreeze $v_{i-1}$, recolour it to avoid $\b(v_i)$, and recolour $v_i$ with $\b(v_i)$.

\textbf{Case 2: We are not in Case 1.}
Since the lists are not identical, there exists a vertex $v_i$ with $L(v_i)\not\subseteq L(v_{i+1})$.
So there exists an $L$-colouring $\c$ with $\c(v_i)\notin L(v_{i+1})$.
By applying Case 1 (twice), we have $\a \sim \c \sim \b$.
\end{proof}

Paths exhibit the same subtlety as cycles. 
To see this intuitively, we view a path and its lists as being formed from a cycle and its (identical) lists, by fixing the colours on one or more successive cycle vertices and deleting them, as well as deleting colours on deleted vertices from the lists of their undeleted neighbours.
Starting from inequivalent $3$-colourings of cycles, we get the same for paths.
This motivates our interest in the following sufficient condition for equivalence.

\begin{lem}
\label{path-lem}
    Let $G=P_n$, for some  $n\ge 3$.
    Fix a \pla{1} $L$ for $G$. 
    If $\a$ and $\b$ are both unfrozen $L$-colourings and $|\cup_{v\in V(G)}L(v)|\ge 4$, then $\a \sim \b$.
\end{lem}

\begin{proof}
    Denote $V(G)$ by $v_1,\ldots,v_n$.
    We assume that $|L(v_i)|\le 3$ for all $i$, or we are done by the Key Lemma.
    So by the hypothesis $|\cup_{v\in V(G)}L(v)|\ge 4$, there exists $i$ such that $L(v_i)\not\subseteq L(v_{i+1})$ and $L(v_{i+1})\not\subseteq L(v_{i+1})$.
    We assume that $\a(v_i)\notin L(v_{i+1})$ and $\a(v_{i+1})\notin L(v_i)$.
    If this is not the case, then we construct such an $L$-colouring $\c$.
    By the case we consider below, $\a\sim \c \sim \b$, as desired.

    We start from $\a$ and recolour, as follows, to reach $\b$.
    By \Cref{move-unfrozen-obs}, we assume that $v_i$ is unfrozen in $\b$.
    As long as $v_i$ is coloured from $\a(v_i)$, vertex $v_{i+1}$ remains unfrozen, so we can use~ $v_{i+1}$ to put $\b(v_j)$ on $v_j$ for all $j\ge i+3$.  (The proof is by induction on $n-j$.)
    After we have $v_j$ coloured with $\b(v_j)$ for all $j\ge i+3$, we recolour $v_{i+1}$ with $\a(v_{i+1})$.
    By swapping the roles of $i+1$ and $i$, we can use vertex $v_i$ to recolour $v_j$ with $\b(v_j)$ for all $j\le i-2$.
    (Equivalently, we relabel each vertex $v_j$ as $v_{n+1-j}$ and repeat the steps above.)
    So the only vertices where our current colouring differs from $\b$ are contained in $\{v_{i-1},v_i,v_{i+1},v_{i+2}\}$.

    By relabeling, we call these vertices $v_1,v_2,v_3,v_4$, and call the current colouring $\a$.
    Recall that~$v_2$ is unfrozen in $\b$.
    If $\b(v_1)\ne \a(v_2)$, then in $\b$ recolour $v_2$ with $\a(v_2)$; now we can recolour~$v_1$ with~$\a(v_1)$ and can recolour $v_3,v_4$ to match $\a$, since $v_3$ is unfrozen.
    Otherwise, we recolour $v_1$ to avoid $\a(v_1)$, first recolouring $v_2$ to unfreeze $v_1$ if needed, and conclude as above.
\end{proof}

The following lemma is easy to verify, essentially by brute force.
But stating it explicitly is convenient for simplifying later proofs.

\begin{lem}
\label{P3-lem}
    Let $G=P_3$ and $L$ be a \pla{1} for $G$.
    If $\a,\b$ are unfrozen $L$-colourings, then $\a\sim\b$.
\end{lem}

\begin{proof}
    Denote $V(G)$ by $v_1,v_2,v_3$.
    If $L(v_1)\not\subseteq L(v_2)$, then we colour $v_1$ from $L(v_1)\setminus L(v_2)$, colour~$v_2$ and $v_3$ as desired, then recolour $v_1$ if needed.
    So assume $L(v_1)\subseteq L(v_2)$; similarly, assume $L(v_3)\subseteq L(v_2)$.
    By symmetry, assume $L(v_1)=\{1,2\}$, $L(v_2)=\{1,2,3\}$, and $L(v_3)=\{1,a\}$ with $a\in \{2,3\}$.
    For each value of $a$, there exist at most $5$ unfrozen $L$-colourings, and we check easily.
\end{proof}

\begin{lem}
\label{claw-lem}
    Let $G$ be formed from the claw $K_{1,3}$ by subdividing a single edge $0$ or more times.
    If~$L$ is a \pla{1} for $G$ and $\a,\b$ are unfrozen $L$-colourings, then $\a\sim \b$.
\end{lem}

\begin{proof}
   Let $v$ be the $3$-vertex, with neighbours $w_1,w_2,w_3$, with $w_1$ and $w_2$ being $1$-vertices, and with~$w_3$ being either a $1$-vertex or $2$-vertex.
    Let $P$ denote the path starting at $v$ and continuing through $w_3$ until we reach a leaf.
    We assume that $v$ is unfrozen in both $\a$ and $\b$.
    We first show that we can assume that $V(P)=\{v,w_3\}$, by showing that $\a\sim \c$ for some colouring $\c$ that agrees with $\b$ on $G-\{v,w_1,w_2,w_3\}$ and has $v$ unfrozen.

    Fix $c\in L(v)\setminus L(w_3)$.
    If $c$ is used on both $w_1$ and $w_2$, then choose a partial colouring $\c$ that agrees with $\b$ on $G-\{v,w_1,w_2,w_3\}$ and agrees with $\a$ on $\{w_1,w_2\}$, and extend it greedily to $G$.
    Applying the Key Lemma to $G-\{w_1,w_2\}$  (in which the size of the reduced list of $v$ is at least two larger than its degree) yields $\a \sim\c$, as desired.
    So instead assume $c$ is used on at most one of~$w_1$ and $w_2$; say $w_1$.
    Since $v$ is unfrozen, we can recolour it if needed to unfreeze $w_1$, then recolour $w_1$, and recolour $v$ with $c$.
    So we assume $v$ is unfrozen under $\a$, and uses colour $c$.
    Next choose a partial colouring $\c$ that agrees with $\b$ on $G-\{v,w_1,w_2,w_3\}$ and agrees with $\a$ on $\{v,w_1,w_2\}$ and extend it to $w_3$.
    Now apply the Key Lemma to $G-\{v,w_1,w_2\}$ (in which the size of the reduced list of $w_3$ is at least two larger than its degree) to obtain $\a \sim \c$.
    In both cases we managed to recolour $\a$ to match $\b$ on $P-\{v,w_3\}$.
    Thus, we assume $w_3$ is a $1$-vertex.
    That is $G=K_{1,3}$.

    Now suppose $L(w_1)\nsubseteq L(v)$.
    We recolour $w_1$ from $L(w_1)\setminus L(v)$, and can recolour $G-w_1$ by \Cref{P3-lem}, afterward recolouring $w_1$ to $\b(w_1)$ if needed.
    So we assume $L(w_i)\subseteq L(v)$ for all $i\in\{1,2,3\}$.
    We also assume $v$ is unfrozen in $\a$.
    Thus, we have $\a(w_i)=\a(w_j)$ for distinct $i,j\in\{1,2,3\}$.
    So we can recolour $v,w_k$, where $k=\{1,2,3\}\setminus\{i,j\}$ so that $w_k$ uses $\b(w_k)$.  
    Afterward, $w_i$ and $w_j$ still have the same colour so $v$ remains unfrozen under $\a$, and we can finish on $G-w_k$ by \Cref{P3-lem}.
    \end{proof}

\begin{lem}
\label{paw-lem}
    Let $G$ be formed from the paw $K_{1,3}+e$ by subdividing the pendent edge $0$ or more times.
    If $L$ is a \pla{1} for $G$ and $\a,\b$ are unfrozen $L$-colourings, then $\a\sim\b$.
\end{lem}

\begin{proof}
    We label the vertices similarly to the proof of the previous lemma; the arguments are also similar, but we will no longer explicitly describe the intermediate colourings $\c$.
    Let $v$ be the $3$-vertex, let $w_1$ and $w_2$ be
    adjacent $2$-vertices, and let $w_3$ be either a $1$-vertex or an interior vertex of a path starting at $v$ and ending at leaf; we call this path $P$.
    We assume that $v$ is unfrozen in both~$\a$ and $\b$.
    We first show that we can assume $V(P)=\{v,w_3\}$; assume not.
    If $L(w_1)\not\subseteq L(v)$, then we recolour $w_1$ from $L(w_1)\setminus L(v)$, and apply the Key Lemma to $G-w_1$ to obtain a colouring that agrees with $\b$ on $P- \{v,w_3\}$, with $v$ unfrozen.
    Thus, we assume that $L(w_i)\subseteq L(v)$ for each $i\in\{1,2\}$. 
    Further, we assume that $L(w_1) \subseteq L(w_2)$, as otherwise we can recolour $w_1$ from $L(w_1)\setminus L(w_2)$ and apply the Key Lemma to $G-w_1$ as before, and then recolour $w_2$ if necessary to unfreeze $v$ again.
    So by symmetry, we assume $L(v)=\{1,2,3,4\}$ and $L(w_1)=L(w_2)=\{1,2,3\}$.
    Since $v$ is unfrozen in $\a$, either $\a(w_3)\notin L(v)$ or else $\a(w_3)=\a(w_i)$ for some $i\in\{1,2\}$.
    In either case, it is easy to recolour $v$ (and possibly $w_j$ with $j=3-i$) so that $v$ is coloured
    from $L(v)\setminus L(w_3)$.
    Now by applying the Key Lemma to $P-\{v\}$, we can recolour it to agree with $\b$ on $P-\{v,w_3\}$.
    If $v$ is frozen after this, we recolour $w_3$ to unfreeze it.
    Thus, we assume that $V(P)=\{v,w_3\}$ as claimed.
    That is, $G=K_{1,3}+e$.

    Recolour $v$, if needed, to unfreeze $w_3$.
    For simplicity, call the present colouring $\a$ (even after the previous paragraph).
    By symmetry between $w_1$ and $w_2$, assume $\a(v)\ne \b(w_2)$.
    Form $\a'$ from~$\a$ by recolouring $w_2$ with $\b(w_2)$ and recolouring $w_1$ arbitrarily.
    By \Cref{clique-lem} (the proof of which does not need the present lemma), we can recolour $\a$ to $\a'$.
    Now restricting to $G-w_2$, we can recolour~$\a'$ to $\b$ by \Cref{P3-lem}, since $w_3$ is unfrozen in $\a'$ and $v$ is unfrozen in $\b$.
\end{proof}

The next lemma (actually, just its non-trivial direction) is a special case of the lemma that follows it.
However, in this case we can give a proof that is independent of the Key Lemma.
This simpler proof leads to a sharper bound on the diameter, which is in fact nearly optimal.

\begin{lem}
\label{clique-lem}
    Let $G=K_n$, with $n\ge 2$, and fix a \pla{1} $L$.
    If $|\cup_{v\in V(G)}L(v)| = n$, then all lists are identical and every $L$-colouring of $G$ is frozen.
    Otherwise, $\a\sim \b$ for all $L$-colourings $\a,\b$.
    In fact, in the latter case, $\dist_{\CLG}(\a,\b)\le 3n/2+2$.
\end{lem}

\begin{proof}
    The `if' statement is trivial, since every colour in $\cup_{v\in V(G)}L(v)$ is used in every $L$-colouring.

    Now we prove the `otherwise' statement.
    Fix $L$-colourings $\a,\b$.
    We build a digraph $\vecD$, with $V(\vecD)=V(G)$ and $\vec{vw}\in E(\vecD)$ if $\b(v)=\a(w)$.
    Note that $\Delta^+(\vecD)\le 1$ and $\Delta^-(\vecD)\le 1$.
    So $\vecD$ consists of directed paths, directed cycles, and isolated vertices.
    A \emph{bad cycle} is a directed cycle in~$\vecD$ for which all colours in the lists of its vertices are used either on other cycles or isolated vertices (which already have the correct colour).
    If a cycle is not bad, then it is \emph{good}.
    We prove the stronger bound $\dist_{\CLG}(\a,\b)\le |E(\vecD)|+\#\text{cycles}(\vecD)+2\times\mathbf{1}$, where $\mathbf{1}$ is the indicator function for $\vecD$ containing a bad cycle.
    Our proof is by induction on $|E(\vecD)|$; the base case $|E(\vecD)|=0$ holds trivially.

    Suppose $\vecD$ contains a good cycle.
    Specifically, there exists a directed cycle $\vecC$ and a colour $c$ and a vertex $v\in \vecC$ such that $c\in L(v)$, but $c$ is not currently used on any vertex of $G$.
    Recolour~$v$ with $c$, which opens $\vecC$ into a directed path, $\vec{P}$.
    Let $x$ be a sink of $\vec{P}$.
    Recolouring $x$ with $\b(x)$ removes one arc from $\vec{P}$; we repeat this process until all vertices of $\vec{P}$ are coloured correctly.
    More broadly, in the same way we handle every good cycle.
    So if $\vecD$ contains a cycle, then we assume it is a bad cycle, and $\mathbf{1}=1$.
    We handle all bad cycles together at the end.

    Suppose $\vecD$ contains a directed path.
    Similar to above, we recolour the sink $x$ of the path with~$\b(x)$, which shrinks the path by an edge, and we handle the shorter path similarly.
    So we assume $\vecD$ contains no paths.
    Note that recolouring good cycles and paths never creates bad cycles.

    Finally, suppose $\vecD$ contains only bad cycles (and isolated vertices).
    Since $|\cup_{x\in V(G)}L(x)| > n$, there exists an isolated vertex $w$ with $c\in L(w)$ that is currently unused on $V(G)$.
    By recolouring~$w$ with $c$ we `unlock' every bad cycle of $\vecD$ (changing it to a good cycle),
    recolour to fix all of these bad cycles, and then recolour $w$ with $\b(w)$.
    The key observation is that, because we unlock every bad cycle at once, we only need to perform this unlocking step at most once.
    So the total number of recolouring steps is at most $|E(\vecD)|+\#\text{cycles}(\vecD)+2\times\mathbf{1}$.
    Since each cycle has length at least~$2$, this gives the claimed bound $3n/2+2$.
    (With a bit more care, we can improve this bound to $(3n+1)/2$, which is sharp.)
\end{proof}

Now we handle the general case of $2$-connected graphs with lists that are not all identical.

\begin{lem}
\label{2connected-lem}
    Let $G$ be a connected graph with a \pla{1} $L$.
    If there exist $v,w\in V(G)$ such that $vw\in E(G)$, $L(v)\not\subseteq L(w)$, and $G-v$ is connected, then $\a\sim\b$ whenever $\a$ and $\b$ are both unfrozen $L$-colourings.
\end{lem}

\begin{proof}
    It suffices to consider the case when $\a(v)\notin L(w)$.
    If not, then we form a new $L$-colouring~$\c$ by first colouring $v$ from $L(v)\setminus L(w)$ and then colouring each other vertex arbitrarily from its list to get a proper $L$-colouring.
    (This is possible since $L$ is a \pla{1}.)
    By the case we consider below, we have $\a \sim \c \sim \b$, as desired.

    By \Cref{move-unfrozen-obs}, we assume that $v$ is unfrozen in $\b$.
    If $|\b^{-1}(\a(v))\cap N(v)|\ne 1$, then let $\b_0:= \b$.
    Otherwise, let $x$ be the unique vertex in $\b^{-1}(\a(v))\cap N(v)$.
    We assume $x$ is unfrozen, by recolouring $v$ if needed; let $\b_0$ be the resulting colouring.

    Form $\b_1$ from $\b_0$ by recolouring $v$ with $\a(v)$ and recolouring each $z\in N(v)\cap \b_0^{-1}(\a(v))$ from $L(z)\setminus \left(\{\a(v)\}\cup \bigcup_{y\in N(z)\setminus\{v\}} \b_0(y)  \right)$.
    (Again, this is possible because $L$ is a \pla{1}.)
    To move from $\a$ to $\b_1$, we let $G':=G-v$, let $L'(z):=L(z)\setminus\{\a(v)\}$ when $z\in N(v)$, and let $L'(z):=L(z)$ otherwise.
    By hypothesis, $G'$ is connected.
    Since $L'$ is a \pla{1} for~$G'$, and $|L'(w)|\ge \deg_{G'}(w)+2$, the Key Lemma gives $\a|_{G'} \sim_{L'}\b_1|_{G'}$; thus $\a \sim_L \b_1$.

    We now consider three cases based on $|\b^{-1}_0(\a(v))\cap N(v)|$: it is at least $2$, it is $0$, or it is $1$.

    If $|\b^{-1}_0(\a(v))\cap N(v)| = 0$, then $\b_0$ differs from $\b_1$ only on $v$.
    Recolouring $v$ in $\b_1$ with $\b_0(v)$ shows that $\b_1\sim \b_0$; hence $\a\sim \b_1\sim \b_0\sim \b$, as desired.

    Suppose instead that $|\b^{-1}_0(\a(v))\cap N(v)| \ge 2$.
    Note that in $\b_1$ vertex $w$ is unfrozen because $\b_1(v)=\a(v)\notin L(w)$.
    So starting from $\b_1$ we recolour $w$, if needed, to unfreeze $v$ and recolour~$v$.
    This makes colour $\a(v)$ available for all $z$ in the independent set $\b^{-1}_0(\a(v))\cap N(v)$.
    After recolouring all such $z$ with $\a(v)$, we recolour $v$, if needed, to avoid $\b_0(w)$; this is possible because $|\b^{-1}_0(\a(v))\cap N(v)|\ge 2$.
    Finally, we recolour $w$ with $\b_0(w)$ and recolour $v$ with $\b_0(v)$.
    Thus, we have $\a \sim \b_1 \sim \b_0\sim \b$ whenever $|\b^{-1}_0(\a(v))\cap N(v)| \ne 1$.

    Finally, assume instead that $|\b^{-1}_0(\a(v))\cap N(v)| =1$.
    As above, we again let $x$ be the unique vertex in $\b^{-1}_0(\a(v))\cap N(v)$.
    Note that $\b_1(z)=\b_0(z)$ for all $z\notin\{v,x\}$.
    Let $G':=G[\{v,w,x\}]$ and $L'(z):=L(z)\setminus \cup_{y\in N(z)\setminus\{v,w,x\}}\b_1(y)$ for all $z\in \{v,w,x\}$.
    Let $\b_i'$ denote the restriction to $G'$ of~$\b_i$ for each $i\in\{1,2\}$.
    So it suffices to show that $\b_0'\sim_{L'}\b_1'$.
    Note that $v$ is unfrozen in $\b_0'$ and $w$ is unfrozen in $\b_1'$.
    If $G'\cong K_3$, then $|L'(z)|\ge 3$ for all $z\in \{v,w,x\}$, but $|\cup_{z\in \{v,w,x\}}L'(z)|\ge 4$ since $\a(v) \in L'(v)\setminus L'(w)$, so by~\cref{clique-lem} we get $\b_1'\sim_{L'}\b_0'$.
    Thus, we assume that $wx\notin E(G)$, so $G'\cong P_3$.
    Hence $|L'(v)|\ge 3$, $|L'(w)|\ge 2$, $|L'(x)|\ge 2$.
    So we are done by \Cref{P3-lem}.
\end{proof}

\subsection{Proof of the Main Theorem}
\label{sec4.2}

In this section, we prove our Main Theorem.
For convenience, we restate it below.

\begin{theorem}
\label{connected-thm}
    Let $G$ be a connected graph with $\Delta\ge 3$, and let $L$ be a \pla{1}.
    If $\a$ and $\b$ are both unfrozen $L$-colourings, then $\a\sim\b$.
    In fact, $\dist_{\CLG}(\a,\b)=O(|V(G)|^2)$.
\end{theorem}

\begin{proof}
    Our proof is by induction on the number of blocks in $G$, and it follows the outline we gave just before \Cref{sec4.1}.
    The base case is when (i) $G$ is formed from $K_{1,3}$ by subdividing a single edge $0$ or more times, (ii) $G$ is formed from $K_{1,3}+e$ by subdividing the pendent edge $0$ or more times, or (iii) $G$ is $2$-connected with $|V(G)|\ge 3$.
    Cases (i) and (ii) are handled by Lemmas~\ref{claw-lem} and~\ref{paw-lem}.
    Case (iii) is handled by either \Cref{regular-thm} or \Cref{2connected-lem}. 
    
    Now suppose that $G$ has a cut-vertex.
    Let $H$ be a leaf in the block tree, with a unique cut-vertex~$v$.
    Let $w$ be a neighbour of $v$ in $H$.

    \setcounter{claim}{0}
    \begin{claim}
    \label{clm1B}
    $L(w)\subseteq L(v)$ and $L(x)=L(w)$ for all $x\in V(H)-v$.
    \end{claim}

    \begin{claimproof}
    Suppose there exist distinct vertices $x_1,x_2\in V(H)-v$ with $L(x_1)\ne L(x_2)$.
    Since $H-v$ is connected, we can assume that $x_1x_2\in E(H)$.
    By symmetry, we assume that $L(x_1)\not \subseteq L(x_2)$.
    But now we are done by \Cref{2connected-lem}.
    Thus, we assume that $L(x_1)=L(x_2)$; hence, $L(x)=L(w)$ for all $x\in V(H)-v$.
    Furthermore, if $L(w)\not\subseteq L(v)$, then we are again done by \Cref{2connected-lem}.
    This proves the claim.
    \end{claimproof}

    \begin{claim}
    \label{clm2B}
    If $H$ is not a cycle, clique, or clique with an edge subdivided (to create vertex $v$), then there exist $x,y\in V(H)-v$ such that $\dist(x,y)=2$ and $H-\{x,y\}$ is connected.
    \end{claim}

    \begin{claimproof}
    By \Cref{clm1B}, and the fact that $H$ is not a cycle, all vertices in $H-v$ have equal degrees, which are at least $3$.
  
    First suppose that there exists $z\in V(H)$ such that $\dist(v,z)\ge 3$.
    Note that~$z$ has non-adjacent neighbours, say $x,y$ (since $H-v$ is connected and all its vertices have the same degree in $G$).
    If $H-\{x,y\}$ is connected, then we are done.
    Assume instead that $H-\{x,y\}$ is not connected; so $y$ is a cut-vertex in $H-x$.
    Thus, $H-x$ has at least two endblocks, and each has a non-cut-vertex adjacent to $x$; call these $x'$ and $y'$.
    Now $H-\{x',y'\}$ is connected (since $\deg(x)\ge 3$),
    and $\dist(x',y')=2$, as desired.

    Assume instead that $\dist(v,z)\le 2$ for all $z\in V(H)-v$.
    Let $V_i:=\{z\in V(H): \dist(v,z)=i\}$ for each $i\in \{1,2\}$.
    And let $H_i:=H[V_i]$.
    We assume that $V_2\ne \varnothing$ (otherwise, we simply let $x,y$ be non-adjacent vertices in $H_1$, which exist because $H$ is not a clique).
    If any vertex of $V_1$ has neighbours in at least two components of $H_2$, then we are done; so assume not.
    Note that each component of $H_2$ has at least two neighbours in $H_1$, since $H$ is $2$-connected.
    Note that each component of $H_2$ must be a clique, or we pick $x,y$ in such a component, and we are done.
    Similarly, if $x\in H_1$ has a neighbour in some component $C$ of $H_2$, then $x$ is adjacent to all vertices of $C$ (or we take $y$ to be a nearest non-neighbour in $H_2$).
    If any vertex $x$ in $V_1$ has no neighbour in $V_2$, then we are done (note that some such $x$ is distance $2$ from some vertex $y\in V_2$, and $H-\{x,y\}$ is connected); so assume not.

    If $H_2$ has at least two components, then there exist $x,x'\in V_1$ adjacent to different components of $H_2$ such that $xx'\in E(H)$ (since $H-v$ is connected).
    So if $H_2$ has at least two components then we are done; consider $x,x'\in V_1$ with neighbours in different components, and pick $y\in V_2$ adjacent to $x'$.
    Now $\dist(x,y)=2$ and $H-\{x,y\}$ is connected.
    Thus, we conclude that $H_2$ has a single component.

    If $H_1$ has at least three vertices, then let $x,y$ be any non-adjacent pair (such a pair exists, since all vertices in $V_1\cup V_2$ have the same degree, but each vertex of $V_1$ is adjacent to $v$).
    So we conclude that $|V_1|=2$.
    But now $H$ is formed from a clique by subdividing a single edge (with the vertices in $V_1$ as its endpoints) to form $v$; this contradicts the hypothesis.
    \end{claimproof}

    \begin{claim}
    \label{clm3B}
    $H$ is an edge, cycle, or a larger clique.
    \end{claim}

    \begin{claimproof}
    First suppose that $H$ is not a clique with an edge subdivided (to create vertex $v$).
    By the previous claim, there exist $x,y\in V(H)-v$ such that $\dist(x,y)=2$ and $H-\{x,y\}$ is connected.
    Let $z$ be a common neighbour of $x$ and $y$.

    We assume that $v$ is unfrozen in both $\a$ and $\b$.
    By \Cref{clm2A}, $L(x)=L(y)$ so there exists a colouring $\c_1$ that uses a common colour $c$ on $x$ and $y$ and agrees with $\a$ on $G-V(H)$.
    Similarly, there exists a colouring $\c_2$ that also uses $c$ on $x$ and $y$ and agrees with $\b$ on $G-V(H)$.

    Let $\a',\b',\c_1',\c_2'$ denote the restrictions to $H$ of $\a,\b,\c_1,\c_2$.
    Form $L'$ from $L$ by restricting to~$H$ and deleting from $L(v)$ all colours used by $\a$ on neighbours of $v$ in $G-H$.
    Since $\a'$ and $\c_1'$ are both unfrozen, by the base case, $\a'\sim_{L'}\c_1'$; thus, $\a\sim_L\c_1$.
    Analogously $\b\sim_L\c_2$.
    Furthermore, $\c_1\sim_L\c_2$, by \Cref{twomatch-lem} applied to the very good pair $x,y$.
    Hence, $\a\sim \c_1\sim \c_2\sim \b$, as desired.

    Finally, assume that $H$ is formed from a clique by subdividing a single edge (to create $v$).
    Let $x,y$ denote the neighbours in $H$ of $v$.
    Since $L(x)=L(y)$ and $N(x)=N(y)$, by possibly recolouring~$y$ we can assume that $\a(x)=\a(y)$ and $\b(x)=\b(y)$.
    So we can form $G'$ from $G$ by identifying $x$ and $y$.
    But now $|L(v)|\ge \deg_{G'}(v)+2$, so we are done by the Key Lemma.
    \end{claimproof}

    \begin{claim}
    \label{clm4B}
        $H$ is an edge.
    \end{claim}

    \begin{claimproof}
    Suppose $H$ is a larger clique (order at least $3$).
    Recall that $v$ is the unique cut-vertex of~$H$,
    and let $u$ be a neighbour of $v$ not in $H$.
    Make $u$ unfrozen in both $\a$ and $\b$.
    Form $\a'$ from $\a$ by recolouring $H$ as in $\b$, except that $v$ keeps it colour and if some neighbour $x$ of $v$ in $H$ has $\b(x)=\a(v)$, then $x$ is coloured arbitrarily (if no such $x$ exists, then choose $x$ to be an arbitrary neighbour in $H$ of $v$).
    Finally, if needed, recolour $u$ so that $|\cup_{y\in H}L_H(y)|>|H|$; to form $L_H(v)$ from $L(v)$, delete all colours used on neighbours of $v$ outside $H$.
       
    Now we can recolour $H$ from $\a$ to~$\a'$.
    Note that $u$ is unfrozen in $\b$ (by assumption) and in $\a'$ since the open neighbourhood $N(u)$
    is coloured identically in $\a'$ and $\a$.
    Let $G':=G-(H-\{v,x\})$.
    Since $\a'$ agrees with $\b$ in $H-\{v,x\}$, we can restrict to $G'$ whenever $\Delta(G')\ge 3$.
    But if $\Delta(G')\le 2$, then we form $G''$ from $G'$ by adding back another neighbour of $v$ (in $H$), and $G''$ is handled by \Cref{paw-lem}.

    Assume instead that $H$ is a $4^+$-cycle.
    Let $v$ be a cut-vertex of $H$ 
    and $u$ be a neighbour of $v$ not in $H$.
    We assume that $u$ is unfrozen in $\a$ and recolour $u$, if needed, to ensure that $|\cup_{x\in H}L_H(x)|\ge 4$; here for $L_H(v)$ we remove from $L(v)$ all colours used on $N(v)\setminus V(H)$, and otherwise, $L_H(x):=L(x)$.
    Form $\a'$ from $\a$ by recolouring $H-v$ as in $\b$, and colouring $v$ arbitrarily to make the colouring proper.
    (Possibly we need $\a'(v)\ne \b(v)$ because $\a'(y)=\a(y)=\b(v)$ for some $y\in N(v)\setminus V(H)$.)
    We also know that $\a'$ is unfrozen on $H$ because $|\cup_{x\in H}L_H(x)|\ge 4$, by construction.
    Thus, we can assume $\a'$ is unfrozen on $N[v]\cap V(H)$; if not, then we `push' this unfrozenness around the cycle (we do this in both $\a'$ and $\b$, since all vertices in $V(H)\setminus N[v]$ have identical colourings of their closed neighbourhoods in $\a'$ and $\b$).
    Let $z$ be an arbitrary neighbour of $v$ in $N[v]\cap V(H)$; further, pick $z$ to be unfrozen (in $\a'$) if $v$ is frozen.
    So we can restrict to $G-(H-\{v,z\})$; this works since in $\b$ vertex $u$ is unfrozen, and in $\a'$ either $v$ or $z$ is unfrozen.
    \end{claimproof}

By the previous claim, every endblock is $K_2$.
Let $P$ be a maximal path, containing some endblock, such that all internal vertices of $P$ have degree $2$ in $G$.
Let $v$ be an endpoint of $P$ such that $\deg_G(v)\ge 3$.
Pick neighbours $w,x$ of $v$, not in $P$, as follows.
If possible, pick $w,x$ to be non-adjacent; otherwise, if possible, pick $w$ to be a cut-vertex.
Let $y$ be the neighbour of $v$ on $P$.

First suppose that $w,x$ are non-adjacent.
Consider the subgraph $J$ induced by $\{w,x\}\cup V(P)$; note that $J$ is formed from $K_{1,3}$ by subdividing one edge $0$ or more times.
To begin, we show that we can assume that $V(P)=\{v,y\}$.
By \Cref{move-unfrozen-obs}, we assume that $w$ is unfrozen in $\a$.
Form $\a'$ from $\a$ by recolouring $V(P)-\{v,y\}$ to agree with $\b$ and recolouring $y$ arbitrarily to give a proper colouring; note that $w$ is also unfrozen in $\a'$.
Now we can recolour $\a$ to $\a'$ by \Cref{claw-lem}, and recurse on $G-(P-\{v,y\})$; that is, we can assume that $V(P)=\{v,y\}$, as desired.

Our goal now is to recurse on $G-y$.
To this end, we form $\a''$ from $\a'$ by recolouring $y$ with~$\b(y)$, recolouring $v$ with a colour not in $L(y)$, and recolouring $w$ and $x$ arbitrarily to get a proper colouring.
Note that $y$ is unfrozen under $\a''$, due to the colour of $v$.
Again by \Cref{claw-lem}, we can recolour from $\a'$ to $\a''$.
If $G-y$ is a path, then $G$ is a subdivision of $K_{1,3}$.
In that case we can apply the above reduction via~\cref{claw-lem} to each of the maximal paths incident to the central vertex, after which we have reduced to the case that the graph is $K_{1,3}$ and we can finish by applying~\cref{claw-lem} one final time.
Thus, we assume that $G-y$ is neither a path nor a cycle (since every end-block is an edge).

As $\a''(y)=\b(y)$, we can recurse on $G-y$ if it contains an unfrozen vertex; so assume it does not.
In particular, we assume that $v,w,x$ are frozen and that $y$ is the only unfrozen vertex.
In this case we will recolour $G[\{v,w,x,y\}]$ to unfreeze one of $v,w,x$.
By symmetry, we assume
that $L(v)=\{1,2,3,4\}$, $\a''(v)=1$, $L(y)=\{2,3\}$, $\a''(y)=\b(y)=2$, $L(w)=\{1,3\}$, $\a''(w)=3$, $L(x)=\{1,4\}$, and $\a''(x)=4$.
Now we recolour $y$ with 3, $v$ with 2, $w$ and $x$ with 1, $v$ with 4, and~$y$ with 2.
Now $v$ is unfrozen, and we used the final two recolourings to ensure that $y$ ends up with colour $\b(y)$ again.
So now we can recurse on $G-y$.

Suppose instead that $w,x$ are adjacent.
First suppose that $w$ is a cut-vertex.
As in the previous paragraph, we let $J:=G[\{w,x\}\cup V(P)]$, and we can assume that $V(P)=\{v,y\}$ by \Cref{paw-lem}.
We recolour $J[\{v,w,x,y\}]$ so that $y$ matches $\b(y)$.
If at least one of $v,w,x$ is unfrozen, then we're done by induction on $G-y$; so assume instead that only $y$ is unfrozen.

Up to renaming colours, we assume that $L(v)=\{1,2,3,4\}$, $\a''(v)=1$, $L(y)=\{2,3\}$, $\a''(y)=2$, $L(w)=\{1,3,4\}$, $\a''(w)=3$, $L(x)=\{1,3,4\}$, and $\a''(x)=4$.
Now we recolour $y$ with 3, $v$ with~2, $w$ with 1, and $x$ with 3.
Now we can get any colours we like (in particular $\b$) on all vertices that are not neighbours of $x$ and lie in the component $H'$ of $G-\{x,y\}$ that contains $v$.
This uses \Cref{twomatch-lem} on the very good pair $x,y$ in $G[V(H')\cup\{x,y\}]$.
Since $w$ is a cut-vertex, this fixes the colouring on at least one additional block.
Similarly, if $H'$ is the block of $G$ containing $v,w,x$, then this approach works if any vertex of $H'$ is a cut-vertex in $G-\{x,y\}$.
Thus, we conclude that $H'$ is an endblock; in fact, $G$ consists only of $H'$ and $P$.
But now $H'$ must be $K_2$, by \Cref{clm4B} above, which is a contradiction. 
This finishes the inductive step and completes the proof.
\end{proof}

\begin{remark}
    We note that the implicit multiplicative constant, call it $C$, in the bound $\diam \CLG=O(|V(G)|^2)$ need not be too large.
    We have not made an effort to optimise $C$, but it is straightforward to check that we can take $C=10$. 
\end{remark}

\subsection{Few Frozen Colourings}

To complete the proof of the Main Theorem, it remains to show that the number of frozen colourings is vanishingly small compared to the total number of proper colourings, provided $\Delta$ is sufficiently small compared to $n$.

\begin{lem}
\label{lem:exp_nonfrozen_swap}
    Consider a \pla{1} $L$ for a connected graph $G$ with $n$ vertices and maximum degree $\Delta$.
    Let $\CLG^f$ denote the collection of frozen $L$-colourings of $G$.
    Then
    \[\frac{|\CLG^f|}{|\CLG|} \leq 2^{-n/\Delta^4},\]
    unless $G$ is a complete graph.
\end{lem}

\begin{proof}
    Consider a set $S=\{v_1w_1,v_2w_2,\ldots, v_sw_s\}$ of edges such that $(N[v_i] \cup N[w_i])_{i\in\{1,\ldots,s\}}$ are disjoint sets of vertices, and $N[v_i] \neq N[w_i]$ for all $i$.
    Given a colouring $c$ and edge $vw$, the \emph{swap of~$c$ at $vw$} is the colouring $c'$ given by $c'(v)=c(w), c'(w)=c(v)$ and $c'(x)=c(x)$ for all $x\notin \{v,w\}$.
    For a subset $M \subseteq S$, the swap of $c$ at $M$ is obtained by swapping at each edge of $M$.
    We claim that if $c$ is frozen and $M \neq \varnothing$, then the swap at $M$ is proper and unfrozen.
    If $M$ consists of one edge $vw$ this follows because $c(N[v])=L(v)$ and $c(N[w])=L(w)$ in a frozen colouring, ensuring the swap of $c$ at~$vw$ remains proper.
    Since $N[v]\neq N[w]$, without loss of generality there is a vertex $x_{vw} \in N[v]\setminus N[w]$, so in the swap of $c$ at $vw$ the colour $c(v)$ is in $L(x_{vw})\setminus N[x_{vw}]$, showing that the swap is unfrozen.
    The fact that the sets $(N[v_i] \cup N[w_i])_{i\in\{1,\ldots,s\}}$ are disjoint ensures that the vertices~$x_{v_iw_i}$ are disjoint and that each of them is adjacent to at most one edge in $S$, so that swapping additional edges of~$S$ preserves their unfrozenness.

    Let $g_M: \CLG^f \rightarrow \CLG$ be the function that maps a frozen colouring $c$ to the colouring swapped at $M$.
    By the above $\{g_M(\CLG^f) : \varnothing \neq M\subseteq S\}$ are disjoint sets of unfrozen colourings, each of size $|\CLG^f|$.
    There are $2^{|S|} -1$ non-empty choices for $M$, so
    \[\frac{|\CLG^f|}{|\CLG|} \leq 2^{-|S|}.\]

    It remains to show that $S$ can be chosen of size at least $n/\Delta^4$.
    For this consider a maximum independent set $I$ of $G^4$.
    This ensures that for any two vertices $v_i,v_j \in I$, the distance between $v_i$ and $v_j$ is at least $5$.
    Note that for every $v_i\in I$, there is an edge $v_iw_i$ of $G$ such that $N[v_i]\neq N[w_i]$ (indeed: if not, then $G$ is the complete graph on $N[v_i]$).
    For every $v_i\in I$ select one arbitrary such edge $v_iw_i$, and include it in $S$.
    Since the distance between $v_i$ and $v_j$ is at least $5$, we have that $N[v_i]\cup N[w_i]$ and $N[v_j]\cup N[w_j]$ are disjoint, for all distinct $v_i,v_j\in I$.
    Since the maximum degree of $G^4$ is less than~$\Delta^4$, it follows that $|S|=|I| \ge n/ \Delta^4$, as desired.
\end{proof}

\section{Sharpness and Open Questions}
\label{sec:sharpness}

It is natural to ask if our Main Results can be strengthened or extended.
To show that we cannot strengthen the Main Theorem by lowering even just one list-size, we first describe examples that imply the Shattering Observation.
After that, we discuss some possible extensions of our Main Results, particularly those for correspondence colouring.

\begin{lem}
\label{basic-shatter-lem}
    There exist connected graphs $G$ with list-assignments $L$ in which $|L(w)|= \deg(w)$ for one vertex $w$ and $|L(v)| \geq \deg(v)+1$ for all other vertices, such that the $L$-recolouring graph has exponentially many non-singleton components.
\end{lem}

\begin{proof}
    Let $K_n$ be the complete graph in which all vertices are assigned the same list  $L(v) =\{1,\ldots,n\}$.
    Then we have $|L(v)| \geq \deg(v)+1$ for all vertices and $\Co{K_n}{L}$ has just $n!$ singleton components.
    Add a new vertex $x$, connect it to one vertex $w$ of $K_n$, and set $L(x) = \{n+1,\ldots,n+p\}$ for some $p \ge 2$.
    The resulting graph and list-assignment satisfies the conditions in the claim, and the $L$-recolouring graph is formed of $n!$ components of size $p$.
\end{proof}

While the simple example given in~\cref{basic-shatter-lem} is very irregular, it is not too difficult to construct $\Delta$-regular examples as well, for any odd $\Delta\geq 5$.
The examples we found use many distinct lists, though.
As a counterpoint to Theorem~\ref{thm:fjp}, it could be interesting to try to strengthen Lemma~\ref{basic-shatter-lem} by restricting to $\Delta$-regular graphs with essentially equal list-assignments: $L(w)=\{1,\dots,\Delta\}$ for some $w$, and $L(v)=\{1,\dots,\Delta+1\}$ for all other $v$.

Note also that the examples in the proof of Lemma~\ref{basic-shatter-lem} have connectivity~$1$.
In fact, we don't know how to construct examples with larger connectivity, which raises the following question.

\begin{question}
\label{que:conn}
    Can the Main Theorem be extended to $k$-connected graphs, if at most $k-1$ vertices~$v$ have lists with size smaller than $\deg(v)+1$?
\end{question}

One observation that supports such an extension is the following.
Suppose that $G$ is a $k$-connected graph with a list-assignment $L$ such that $|L(w)|\le\deg(w)$ for at most $k-1$ vertices $w\in V(G)$, and $|L(v)|\ge\deg(v)+1$ for all other $v\in V(G)$.
Then if we remove the vertices with small lists, the remaining graph is a connected graph satisfying the condition in the Main Theorem, but it is not obvious what this means for the structure of $\CLG$.

Another possible strengthening of our main results involves reducing the upper bound of the diameter of the $L$-recolouring graph.
Since the diameter $\diam(\Co{P_n}{3})$ of the $3$-recolouring graph of the path $P_n$ is $\Theta(n^2)$, a quadratic bound is best possible in general.
However, is the presence of long induced paths of degree-$2$ vertices necessary for the $L$-recolouring graph to have such a large diameter?
We propose the following conjectures. 

\begin{conj}
\label{conj:diameter}
    Let $G$ be a connected graph with $n$ vertices, and let $L$ be a list-assignment of~$G$.
    \begin{list}{}{%
    \setlength{\itemsep}{0pt}
    \setlength{\topsep}{2pt}
    \setlength{\labelsep}{0.5\parindent}
    \setlength{\labelwidth}{\parindent}
    \setlength{\leftmargin}{1.5\parindent}
    \setlength{\itemindent}{0pt}
    \setlength{\listparindent}{0pt}}
    \item[\rm (a)\hss]\label{conj:keylemma} If, in addition to the conditions in the Key Lemma, $G$ has minimum degree at least $3$, then $\CLG$ is connected and has diameter $O(n)$.
    \item[\rm (b)]\label{conj:mainth} If, in addition to the conditions in the Main Theorem, $G$ has minimum degree at least $3$, then $\hCLG$ is connected and has diameter $O(n)$.
    \item[\rm (c)]\label{conj:mainths} If, in addition to the conditions in the Main Theorem, $G$ has no path of degree-$2$ vertices of length more than $t$ for some $t\ge0$, then  $\hCLG$ is connected and has diameter $O((t+1)n)$.
    \end{list}
\end{conj}

Regarding Conjecture~\ref{conj:keylemma}(a), note that \cref{lem:degreeplus_linearbound} provides a bound on the diameter of order $\overline{d}n$, where $\overline{d}$ is the average degree.
As a first step towards part~(b) of the conjecture,
recall the result of Bousquet et al.~\cite{BFHR} that $\diam(\Co{G}{\Delta})\le f(\Delta)\cdot n$ for standard colouring.
But their function $f$ is superexponential in $\Delta$, so proving part~(b) for standard colouring would already be a significant improvement over what is known.

\subsection{Correspondence Colouring}

As alluded to earlier, in most of our arguments towards the Key Lemma and the Main Theorem, it is not so relevant which specific colours are present in which list.
More relevant are the specific list-sizes as compared to their respective vertices' degrees.
To what extent are the colours themselves essential to the Main Results?

This question leads us naturally to a strengthened notion, one that has been intensively studied in chromatic graph theory in recent years, 
called correspondence colouring (also called DP-colouring)~\cite{DP}.
In this closing part of the paper, we discuss how we might be able (or not) to extend our results to this more general context.

To introduce the notion informally, suppose that we think of list colouring as ordinary colouring but under some adversarial choice of worst-case assignments of colour-lists at each vertex.
Analogously, correspondence colouring can be thought of as giving our adversaries even more power by letting them choose worst-case matchings of `colours' between the lists of adjacent vertices --- any pair of matched `colours' may not be used simultaneously in order for the colouring to be `proper'.

There are various equivalent definitions and notations, so we fix our formalisms for correspondence colouring as follows.
A (\emph{correspondence}-)\emph{cover} of a graph $G$ is a pair $(L,H)$, where $H$ is a graph and $L: V(G) \to 2^{V(H)}$ are such that the following hold:
    \begin{list}{}{%
    \setlength{\parsep}{0pt}
    \setlength{\topsep}{2pt}
    \setlength{\labelsep}{2mm}
    \setlength{\labelwidth}{1.0\parindent}
    \setlength{\leftmargin}{\parindent}\setlength{\itemindent}{0pt}
    \setlength{\listparindent}{0pt}}
    \item[$\bullet$\hss] the sets $\{L(v):v\in V(G)\}$ partition $V(H)$;
    \item[$\bullet$\hss] if there is an edge in $H$ between $L(u)$ and $L(v)$, then $uv$ is an edge of $G$;
    \item[$\bullet$\hss] for any $v\in V(G)$, $L(v)$ induces a clique in $H$;
    \item[$\bullet$\hss] for any $u,v\in V(G)$, the bipartite subgraph of $H$ induced between $L(u)$ and $L(v)$ is a matching.
\end{list}
(A cover is analogous to a list-assignment $L$.)
Given $k: V(G)\to \mathbb{N}$, we call the cover \emph{$k$-fold} if $|L(v)| \ge k(v)$ for all $v\in V(G)$.
(This is analogous to a lower bound condition on the list-sizes for~$L$.)
An $(L,H)$-colouring is an independent set of $H$ of size $|V(G)|$. 
(This is analogous to a proper $L$-colouring.)

Many prominent results for list colouring have been generalised to correspondence colouring as well.
But as our next example shows, this fails for our Main Theorem, even for cliques.

\begin{example}
\label{clique-example}
    Denote the vertices of $K_n$ by $v_1,\ldots,v_n$.
    Form the cover graph $(L,H)$ by connecting $(v_i,k)$ and $(v_j, \ell)$ for every $i \ne j$ if and only if $k=\ell \in \{3,\ldots,n\}$ or $k=1, \ell=2$ or $k=2, \ell=1$.
    That is, using a colour on a vertex forbids that same colour on each neighbour, except for colours~1 and~2.
    Using colour 1 on any vertex forbids 2 on its neighbours, and vice versa.
    Now $\Co{K_n}{(L,H)}$ has exactly two large components.
\end{example}

\begin{proof}
    By the pigeonhole principle, every $(L,H)$-colouring of $K_n$ has at least two vertices coloured with 1 or 2 (and these must be the same colour).
    So we have two components (with no frozen colourings), one containing the `all 1' colouring, and one containing the `all 2' colouring.
\end{proof}

Being a bit more precise, we can find $n-2$ exceptions.
Starting from a `straight' cover graph of~$K_n$ with lists of size $n$, we `twist' 1 and 2 on each edge in a subclique $K_q$ for some $q\in\{2,\ldots,n-1\}$.
(Note that we get isomorphic cover graphs by letting $q=n-1$ or $q=n$.)

\begin{example}
\label{exp:Kncorrespondence}
\label{clique-example-2}
    Denote the vertices of $K_n$ by $v_1,\ldots,v_n$.
    For a value $q\in\{2,\ldots,n-1\}$, form the cover graph $(L,H)$ by connecting $(v_i,k)$ and $(v_j, \ell)$ exactly when either
    \begin{list}{}{%
    \setlength{\parsep}{0pt}
    \setlength{\itemsep}{1pt}
    \setlength{\topsep}{2pt}
    \setlength{\labelsep}{2mm}
    \setlength{\labelwidth}{1.0\parindent}
    \setlength{\leftmargin}{\parindent}\setlength{\itemindent}{0pt}
    \setlength{\listparindent}{0pt}}
    \item[$\bullet$\hss] $k=\ell$ and $k \in\{3,\ldots,n\}$; or
    \item[$\bullet$\hss] $\max\{i,j\} \le q$ and $\{k,\ell\}=\{1,2\}$; or
    \item[$\bullet$\hss] $\max\{i,j\} > q$ and $k=\ell$.
    \end{list}
    Then $\Co{K_n}{(L,H)}$ has exactly two non-singleton components (as well as many frozen colourings if $q<n-1$).
\end{example}

\begin{proof}
    Consider an $(L,H)$-colouring $\a$.
    If $\a(v_i)\notin\{1,2\}$ for every $1 \le j \le q$, then all edges incident to $\a(v_i)$ in the correspondence-cover are straight for all $i\in\{1,\ldots,n\}$, so $\a$ is frozen.
    We define the sets~$A_i$ of colourings by $A_i=\{\a:\a(v_j)=i \text{ for some } j\le q \}$ for each $i\in\{1,2\}$.
    Note that $A_1\cap A_2 =\varnothing$ and $A_1\cup A_2$ contains all non-isolated vertices of $\Co{K_n}{(L,H)}$.
    It is easy to check that each $A_i$ induces a connected subgraph of $\Co{K_n}{(L,H)}$.
    If for a colouring in $A_i$ exactly one vertex~$v_j$ ($j \le q$) is coloured by $i$, then some $v_k$, for $k>q$, is coloured by $3-i$ and one cannot recolour $v_j$.
    In this case, the only valid recolouring step is to recolour $\a(v_{\ell})$ with $i$ for some $\ell \in\{1,\ldots,q\}\setminus \{j\}$.
    Thus~$A_1$ and $A_2$ are not connected.
\end{proof}

For $n=4$, in addition to the cases $q=2$ and $q=3$ in \Cref{clique-example-2}, we discovered a third $4$-fold correspondence-cover of $K_4$ for which $\Co{K_n}{(L,H)}$ has two non-singleton components.
Denote the vertices of $K_4$ by $v_1,\ldots,v_4$.
The cover graph $(L,H)$ is depicted in~\cref{fig:thirdexample_MobiusKantor}.
It has $L(v_i)=\{(v_i,j) : j \in\{1,2,3,4\}\}$ for every $i \in\{1,2,3,4\}$, and is obtained by adding the five edges
\[\{\,((v_1,0),(v_2,1)),\,((v_1,1),(v_2,0)),\, ((v_2,0),(v_3,2)),\,((v_2,2),(v_3,3)),\,((v_2,3),(v_3,0))\,\},\]
and the straight edges (edges of the form $(v_h,j),(v_i,j)$ for $i \not=h$) elsewhere such that every two lists $L(v_i)$ and $L(v_h)$ span a perfect matching.
Now $\Co{K_n}{(L,H)}$ has exactly two non-singleton components (each containing $24$ $(L,H)$-colourings).
The cover graph is obtained from the M\"{o}bius-Kantor graph, superimposed with the four list-cliques.

By computer we verified that for no other $4$-fold cover graphs does $\Co{K_4}{(L,H)}$ have multiple non-singleton components.
Those examples above suggests that a precise characterisation of `bad' cover graphs for $K_n$ may be hard.

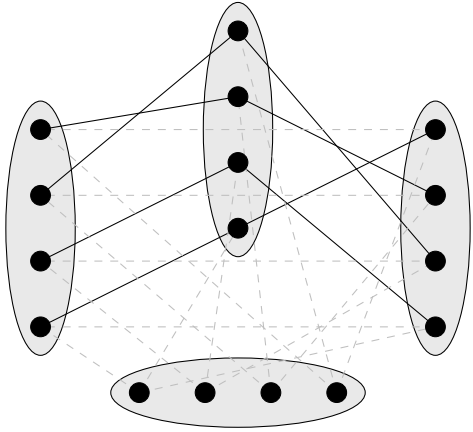
\begin{figure}[ht]
    \centering
    \begin{tikzpicture}[scale=0.875]

\foreach \x in {0,180}{
\draw[fill=gray!17] (\x:3) ellipse (15pt and 55pt);
}

\draw[fill=gray!17] (90:1.5) ellipse (15pt and 55pt);

\draw[fill=gray!17] (270:2.5) ellipse (55pt and 15pt);

 \foreach \r/\y in {-1.5/-1.5,-0.5/-0.5,0.5/1.5,1.5/0.5}{
 \draw(-3,\r)--(0,1.5+\y);
 }

\foreach \r/\y in {-1.5/1.5,-0.5/-1.5,1.5/-0.5,0.5/0.5}{
 \draw (0,1.5+\r)--(3,\y);
 }

\foreach \y in {-1.5,-0.5,0.5,1.5}{

\draw[gray!50, dashed] (\y,-2.5)--(3,\y);
\draw[gray!50, dashed] (\y,-2.5)--(-3,\y);
\draw[gray!50, dashed] (3,\y)--(-3,\y);
\draw[gray!50, dashed] (\y,-2.5)--(0,1.5+\y);
 }

\foreach \y in {-1.5,-0.5,0.5,1.5}{
\draw[fill] (-3,\y) circle (0.15);
\draw[fill] (3,\y) circle (0.15);
\draw[fill] (0,1.5+\y) circle (0.15);
\draw[fill] (\y,-2.5) circle (0.15);
}

 \end{tikzpicture}
    \caption{A $4$-fold correspondence-cover $(L,H)$ of $K_4$ for which $\hCo{K_4}{(L,H)}$ is not connected.}
    \label{fig:thirdexample_MobiusKantor}
\end{figure}

Despite the previous examples, we remark that the Key Lemma does hold for correspondence colouring, with a slightly modified proof.
In the proof of \Cref{lem:recolouringToUniqueBadNeighbour}, the bad neighbours need not form an independent set.
So when we recolour $w$ to unfreeze the frozen bad neighbours, and recolour the first bad neighbour, this could possibly freeze another bad neighbour.
Fortunately, if~$v$ has at least two bad neighbours, then $v$ is unfrozen (even for correspondence) so we can recolour~$v$ to unfreeze its bad neighbours.
(We also handle these bad neighbours of $w$ in order of non-increasing distance from $v$, to avoid undoing progress that we made previously.)
We may need a few more steps than for the list colouring case, but some version of the lemma still holds.

We don't have a good idea of what results might be possible for correspondence colouring,
similar to those for list colouring.
So to conclude the paper, we pose a few questions that might lead us to a better understanding.

\begin{question}\mbox{ }
    \label{corr-open-probs}
    \begin{list}{}{%
    \setlength{\itemsep}{0pt}
    \setlength{\topsep}{2pt}
    \setlength{\labelsep}{0.5\parindent}
    \setlength{\labelwidth}{\parindent}
    \setlength{\leftmargin}{1.5\parindent}
    \setlength{\itemindent}{0pt}
    \setlength{\listparindent}{0pt}}
    \item[\rm (a)\hss] What characterises the `bad' covers of the clique~$K_n$?
    When $n=3$ we have only three non-isomorphic covers, one of which is equivalent to the case of standard colouring.
    Exactly one of these is bad.
    When $n=4$ we have $75$ non-isomorphic covers; a computer search showed that exactly three are bad.
    
    \item[\rm (b)] What is the best possible diameter bound for the Key Lemma with correspondence colouring? 
    Is it $O(n)$ (as we conjecture for list colouring in~\cref{conj:diameter}~(a))? 
    \item[\rm (c)] Does some version of the Main Theorem hold for correspondence colouring, possibly with additional well-understood exceptional graphs?
    \end{list}
\end{question}

\subsection*{Acknowledgements}
The research for this paper was started during a visit of SC, WCvB, and JvdH to the Korteweg--de Vries Institute.
We like to thank the institute, and in particular the members of the Discrete Mathematics \& Quantum Information group for their hospitality and the friendly and inspiring atmosphere; and for the cookies, of course.

\subsection*{Open access statement}
For the purpose of open access, a CC BY public copyright license is applied to any Author Accepted Manuscript (AAM) arising from this submission.

{\small{
\bibliographystyle{habbrv}
\bibliography{reconfiguration}
}}

\end{document}